 \documentclass[11pt, a4paper]{article}
 
\usepackage[a4paper, left=3cm, right=3cm, top=3.5cm, bottom=3.5cm, showframe=false]{geometry}
\usepackage[fleqn]{amsmath}
\usepackage[affil-it]{authblk}
\usepackage{amsthm}
\usepackage{amssymb} 
\usepackage{graphicx}
\usepackage{bbm} 
\usepackage{here}
\usepackage{hyperref}  
\usepackage{subfig}
\usepackage{dsfont}
\usepackage{color} 
\usepackage[usenames,dvipsnames]{xcolor}

\usepackage{setspace}

\colorlet{linkcolour}{blue!70!black}
\colorlet{urlcolour}{magenta}
\hypersetup{colorlinks=true, linkcolor=linkcolour, citecolor=linkcolour, urlcolor=urlcolour}

\usepackage{mathtools}
\usepackage[round]{natbib}
 
\newcommand{\cE}{{\cal E}}

\newcommand{\cM}{{\cal M}}

\newcommand{\cR}{{\cal R}}

\newcommand{\mN}{\mathbb{N}}
\newcommand{\mR}{\mathbb{R}}

\newcommand{\Cov}{{\rm Cov}}

\newcommand{\Var}{{\rm Var}}

\theoremstyle{plain}
\newtheorem{theorem}{Theorem}[section]
\newtheorem{proposition}[theorem]{Proposition}
\newtheorem{corollary}[theorem]{Corollary}
\newtheorem{lem}[theorem]{Lemma}

\theoremstyle{definition}

\newtheorem{assumption}[theorem]{Assumption}
\newtheorem*{assA}{Assumption A1}{\bf}{\it}

\theoremstyle{remark}
\newtheorem{remark}[theorem]{Remark}
\newtheorem{example}[theorem]{Example}

\DeclareMathOperator*{\argmax}{arg\,max}

\numberwithin{equation}{section} 
 
\makeatletter
\def\@maketitle{%
  \newpage
  \null
  \vskip 2em%
  \begin{center}%
  \let \footnote \thanks
    {\Large\bfseries \@title \par}%
    \vskip 1.5em%
    {\normalsize
      \lineskip .5em%
      \begin{tabular}[t]{c}%
        \@author
      \end{tabular}\par}%
    \vskip 1em%
    {\normalsize \@date}%
  \end{center}%
  \par
  \vskip 1.5em}
\makeatother

\begin{document}

\title{Panel data segmentation under finite time horizon} 

\author{Leonid Torgovitski\thanks{E-mail: \texttt{ltorgovi@math.uni-koeln.de}} \thanks{Research partially supported by the Friedrich Ebert Foundation, Germany.}}
\affil{Mathematical Institute, University of Cologne,\\ Weyertal 86-90, 50931, Cologne, Germany}
\date{}

\maketitle

\begin{abstract}
We study the nonparametric change point estimation for common changes in the means of panel data. 
The consistency of estimates is investigated when the number of panels tends to infinity but the sample size remains finite. 
Our focus is on weighted denoising estimates, involving the group fused LASSO, and on the weighted CUSUM estimates. 
Due to the fixed sample size, the common weighting schemes do not
guarantee consistency under (serial) dependence and most typical weightings
do not even provide consistency in the i.i.d. setting when the noise is too dominant.

Hence, on the one hand, we propose a consistent covariance-based extension of existing weighting
schemes and discuss straightforward estimates of those weighting schemes. The performance will be demonstrated empirically in a simulation study. 
On the other hand, we derive sharp bounds on the change to noise ratio that ensure consistency in the i.i.d. setting for classical weightings.

\paragraph{Keywords:}%\hfill\\ 
Panel data, Change point estimation, Segmentation, Nonparametric, CUSUM, Total variation denoising, LASSO, Serial dependence

\end{abstract}

\footnotetext[1]{\copyright ~2015, Elsevier. Licensed under the Creative Commons Attribution-NonCommercial-NoDerivatives 4.0 International \url{http://creativecommons.org/licenses/by-nc-nd/4.0/}} 
\footnotetext[2]{The final article is published under \url{http://dx.doi.org/10.1016/j.jspi.2015.05.007}}

\newpage
\section{Introduction}

The aim of this paper is to study the estimation of changes in the context of {\it panel data}. We focus on {\it common changes}, 
i.e. changes that occur simultaneously in many panels (but not necessarily in all) at the same time points and we consider an asymptotic framework 
where the number ~$d$ ~of panels tends to infinity but the panel sample size ~$n$ ~is fixed.

The analysis of change point estimation in panel data is subject of intensive research (in particular in econometrics) and, 
as discussed in \citet{Bai2010}, dates back at least to the works of \citet{wolfson1992,wolfson1993}. 
However, the setting ~$d\rightarrow\infty$, which we are looking at, is generally not studied much in the literature concerning change point analysis 
and the settings ~$n\rightarrow\infty$ ~or ~$n,d\rightarrow\infty$ ~are far more established.

For the classical setting of ~$n\rightarrow\infty$ ~we refer to \citet{horvath1997limit}. 
In the context of panel data especially the setting ~$n,d\rightarrow\infty$ ~is quite popular (cf., e.g., \citet{Bai2010}, \citet{huskova2012} and \citet{kim2014}). 
Nevertheless, the assumption ~$d\rightarrow \infty$ ~and ~$n$ ~fixed is also quite natural (cf., e.g., \citet{Bai2010}, \citet{vert2011a}, \citet{rao2012} and also \citet{pesta2015}). 
It reflects the situation where the amount of panels, i.e. the dimensionality, is much larger than the sample size.

\citet{Bai2010} and \citet{vert2011a} mention important applications in finance, biology and medicine where in particular the framework of common changes is appropriate: 
In finance such changes may occur simultaneously across many stocks e.g. due to a credit crisis or due to tax policy changes. 
In biology and medicine relevant applications are in the study of genomic profiles within classes of patients. 
As mentioned in \citet{vert2011a} the latter example fits particularly well in the ~$n$ ~fixed and ~$d\rightarrow\infty$ ~framework because the length of panels in genomic studies 
is fixed but the amount of panels can be increased by raising the number of patients.

The body of literature related to change point estimation (and detection) is huge. Hence, we do not attempt to summarize it here and refer the reader instead to the reviews in \citet{jandhyalaa2013}, \citet{aue2013}, \citet{frick2012} and \citet{rice2014}. 
Change point analysis in the ~$d\rightarrow\infty$ ~and ~$n$ ~fixed setting goes at least back to the (aforementioned) papers by \citet{vert2010,vert2011a} and by \citet{Bai2010}. 
Therein estimation of common changes is studied independently from different perspectives. 
However, as we will see, the setups of \citet{vert2010,vert2011a} and of \citet{Bai2010} are closely related\footnote{Notice that \citet{vert2011a} is a revised version of \citet{vert2010}. 
Hence, we will mostly refer to the more recent article.}.

\citet{Bai2010} considered a {\it least squares} estimate for independent panels of linear time series under a single change point assumption and  \citet{vert2011a} 
developed a weighted {\it total variation denoising} approach for the multiple change point scenario. 
Furthermore, \citet{vert2011a} proposed a computationally efficient algorithm and implemented it in a convenient 
MATLAB package GFLseg\footnote{Download is available at \url{http://cbio.ensmp.fr/GFLseg} and is licensed under the GNU General Public License.} 
which we also used in some of our simulations.
\\
\\
In this article we study consistency properties, in particular what we define as {\it perfect estimation}\footnote{See Subsection \ref{sec:theoretical} and \eqref{eq:cons_est_def} below.}, 
for the denoising estimate and for the weighted CUSUM (cumulative sums) estimate under weak dependence.  Both types of estimates depend on certain weighting schemes ~$w$. 
Two schemes,  ~$w^\text{simple}$ ~and ~$w^\text{standard}$, were already considered by \citet{vert2011a} for the denoising 
approach in the ~$n$ ~fixed and ~$d\rightarrow\infty$ ~setting (cf. Subsection \ref{sec:common_schemes} for the precise definition). 
They showed that ~$w^\text{standard}$ ~ensures perfect estimation and therefore has better consistency properties for ~$d\rightarrow\infty$ ~than ~$w^\text{simple}$ ~does. 
(Notice that \citet{Bai2010} showed perfect estimation for the least squares estimate, which corresponds to the weighted CUSUM estimate with ~$w^\text{standard}$.)

 We pick up the ideas of \citet{vert2011a} and extend them in various directions which will shed some new light on weighting schemes in general.  
 First, we will emphasize the connection between the total variation denoising approach and the weighted CUSUM estimates. 
 Notice that \citet{vert2011a} assumed independent panels of independent Gaussian observations. 
 We continue by showing that their consistency results hold true under much weaker distributional assumptions, e.g. for panels of non-Gaussian time series with common factors. 
 This is important since many datasets are neither Gaussian nor independent. 
 An implication of our results is that ~$w^{\text{standard}}$ ~generally does not provide consistency for panels of time series and therefore does not ensure perfect estimation under dependence.

 As a solution, we propose a modified weighting scheme ~$w^{\text{exact}}$, which is a generalization of ~$w^{\text{standard}}$, that takes the covariance structure within panels into account. 
 We show that this is the only choice that may generally ensure perfect estimation and derive quite mild conditions under which ~$w^{\text{exact}}$ ~indeed ensures this property. 
 In a detailed simulation study we confirm our results and demonstrate the gain in accuracy of ~$w^{\text{exact}}$. Moreover, we show that our approach outperforms the classical schemes 
 even in random change point settings and for rather moderate dimensions. In practice, the weights ~$w^{\text{exact}}$ ~have to be estimated. 
 Therefore, we discuss feasible approaches and show their applicability in simulations.

Complementary to the study of perfect estimation, we investigate consistent estimation for a further class of weights ~$w^\text{weighted}$, which contains ~$w^{\text{simple}}$ ~and  ~$w^{\text{standard}}$ ~as special cases, 
and characterize changes which are (not) correctly estimated as ~$d\rightarrow\infty$.

\subsection{Basic setup}
We observe ~$d$ ~panels ~$\{Y_{i,k}\}_{i=1,\ldots,n}$ ~for ~$k=1,\ldots,d$ ~in a {\it signal plus noise} model where
\begin{equation}\label{eq:basicmodel_panel}
	Y_{i,k}=m_{i,k} + \big(\varepsilon_{i,k} + \gamma_k\zeta_i\big).
\end{equation}
Here, ~$\{m_{i,k}\}_{i,k\in\mN}$ ~is an array of deterministic signals and ~$\{\varepsilon_{i,k}\}_{i,k\in\mN}$ ~is an array of random centered noises. 
The ~$\{\zeta_i\}_{i\in\mN}$ ~are the so-called {\it common factors} which are assumed to be random, centered and independent of ~$\{\varepsilon_{i,k}\}_{i,k\in\mN}$. 
Their effect on the ~$k$-th panel is quantified via the deterministic {\it factor loadings} ~$\gamma_k\in\mR$.

We assume a (multiple) common change points scenario given by
\begin{equation}\label{eq:basicmodel}
	m_{i,k}=
\begin{cases}
\mu_{1,k}, & i=1,\ldots, u_1,\\
\mu_{2,k}, &  i= u_1+1,\ldots, u_2,\\
\dots, & \dots,\\
\mu_{P+1,k}, & i= u_{P}+1,\ldots, n,
\end{cases}
\end{equation}
where we call ~$u_1,\ldots, u_P\in \mN$ ~change points. The ~$\mu_{j,k}\in\mR$, $j=1,\ldots,P+1$, describe the piecewise constant signals in each panel, 
i.e. the means of the observations. In other words the means jump simultaneously from levels ~$m_{u,k}$ ~to levels ~$m_{u+1,k}$ ~in all panels ~$k=1,\ldots,d$ ~at change points ~$u\in\{u_1,\ldots, u_P\}$. 
However, we do not require ~$m_{u,k}\neq m_{u+1,k}$ ~to hold for all ~$k=1,\ldots,d$, i.e. the changes do not have to occur in all panels. 
Later on we will impose more specific conditions on the average magnitude of changes.

Subsequently, we assume that ~$n\geq 3$ ~since otherwise the model \eqref{eq:basicmodel} is not reasonable because for ~$n=1$ ~the model may not contain any change and for ~$n=2$ ~it holds trivially that ~$P=1$ ~with ~$u_1=1$.

\subsection{Notation}
We follow the compact matrix notation of \citet{vert2011a} and represent the model \eqref{eq:basicmodel} as
\[	
	Y = \cM + E,
\]
with a deterministic matrix ~$\cM$ ~of means with ~$\cM_{i,k}=m_{i,k}$ ~and a random matrix of errors ~$E$ ~with ~$E_{i,k}=\varepsilon_{i,k}+\gamma_k\zeta_i$. Now, let ~$X$ ~be any ~$n\times d$ ~matrix. 
To shorten the notation we write ~$X_{\bullet,j}$ ~for ~$[X_{1,j},\ldots,X_{n,j}]^T$ ~and ~$X_{i,\bullet}$ ~for ~$[X_{i,1},\ldots,X_{i,d}]$. 
For example ~$Y_{\bullet,k}$ ~represents the ~$k$-th panel and a common change at ~$u$ ~corresponds to
\begin{equation}\label{eq:def_delta}
	\Delta =\cM_{u+1,\bullet}-\cM_{u,\bullet}\neq 0. 
\end{equation}
$\|\cdot\|_F$ ~denotes the Frobenius norm and ~$\|\cdot\|_2$ ~stands for the Euclidean norm. We simply write ~$\|\cdot\|$ ~for the former when no confusion is possible, unless it is stated otherwise.

We will consider functions ~$f(i)$ ~with a discrete support ~$i=1,\ldots,n-1$ ~and say that a function ~$f$ ~is convex (or concave) if this holds true 
for the linear interpolation of points ~$f(i)$ ~on the interval ~$[1,n-1]$. Subsequently, we mean by ~$\argmax$ ~the whole set of points at which the maximum is attained.
\\
\\ 
The paper is organized as follows. In Section \ref{sec:segmentation} we discuss segmentation of panel data. In Subsection \ref{sec:totvardenois} we introduce the concept of the denoising segmentation approach 
in general and then we turn to the single change point scenario in Subsection \ref{sec:single_change_point}. 
First, we clarify the selection of a certain regularization parameter and continue to discuss the relation to a class of weighted CUSUM estimates in Subsection \ref{sec:denoising_cusum}. 
Common weighting schemes are presented in Subsection \ref{sec:common_schemes}. In Subsection \ref{sec:theoretical} we analyze the segmentation procedures with respect to different 
weighting schemes and propose a generalization of existing approaches. Subsequently, we discuss estimates of the generalized weighting scheme in Subsection \ref{sec:estimation}. In Section \ref{sec:simulations} we confirm our theoretical results in a simulation study. Finally, we provide a short summary of the paper in Section \ref{sec:conclusion} and all proofs are postponed to Section \ref{sec:proofs}. 

%\newpage
\section{Segmentation of panel data}\label{sec:segmentation}
We start with a description of the denoising approach to change point estimation of \citet{vert2011a}. 
For an overview of the related literature we refer to the references therein.

\subsection{Total variation denoising estimates}\label{sec:totvardenois}
The total variation denoising approach to segmentation is to solve the convex 
minimization problem\footnote{The objective function in \eqref{eq:totvar} is strictly convex, as a sum of convex functions and due to the strict convexity of the mapping ~$U\mapsto \|Y-U\|^2_F$. 
Moreover, we may restrict the minimization to a compact subset. Therefore, a unique solution exists for any ~$\lambda\geq0$.}
\begin{equation}\label{eq:totvar}
\underset{U\in\mR^{n\times d}}{\text{Minimize}} \;\frac{1}{2}\|Y-U\|^2_F + \lambda \times \text{totvar}(U)
\end{equation}
for an appropriate regularization parameter ~$\lambda \geq 0$ ~under a weighted total variation penalty term
\begin{equation}\label{eq:totvar_penalty}
\text{totvar}(U)=\sum_{i=1}^{n-1}\frac{\|U_{i+1,\bullet}-U_{i,\bullet}\|_2}{w(i,n)}
\end{equation}
with positive, {\it position dependent} weights ~$w(i,n)>0$. We denote the solution of \eqref{eq:totvar} by ~$\hat{U}(\lambda)$ ~and each column ~$\hat{U}_{\bullet,k}$ ~represents the best 
piecewise constant fit to the panel ~$Y_{\bullet,k}$ ~with respect to \eqref{eq:totvar}. Each change in those fits, in the sense of ~$\hat{U}_{u+1,\bullet}\neq \hat{U}_{u,\bullet}$, ~is therefore 
assumed to identify a common change across panels at time point ~$u$. Hence, the set ~${\cal E}$ ~of estimated change points is given by
\begin{equation}\label{eq:set_changes}
	{\cal E}(\lambda) = \left\{u \;|\; \hat{U}_{u,\bullet}(\lambda)\neq \hat{U}_{u+1,\bullet}(\lambda)\right\}.
\end{equation}
The penalty term ~$\text{totvar}(U)$ ~is designed in such a way that ~$\hat{U}(\lambda)$ ~has for ~$\lambda>0$ ~a tendency to reduce the cardinality of \eqref{eq:set_changes}, 
i.e. to reduce the amount of identified change points. Hence, ~$\cE$ ~has a tendency to become smaller as ~$\lambda$ ~increases.

Two extreme cases give some insight: For ~$\lambda\uparrow\infty$ ~the penalty term ~$\text{totvar}(U)$ ~dominates the minimization 
and forces the minimizer ~$\hat{U}$ ~to be constant across rows, 
i.e. we obtain ~${\cal E}=\emptyset$ ~and no change points are identified by this procedure at all. 
In contrast to this, if ~$\lambda=0$ ~then ~$\hat{U}(0)=Y$ ~and therefore ~${\cal E}(0)=\left\{u \;|\; Y_{u,\bullet}(\lambda)\neq Y_{u+1,\bullet}(\lambda)\right\}$, 
i.e. the number of estimated changes corresponds to the number of different consecutive rows of ~$Y$. 
Hence, if e.g. all rows are unique then each point ~$i=1,\ldots,n-1$ ~is identified as a change point. 

\subsection{Single change point scenario}\label{sec:single_change_point}
Following \citet{vert2011a} we will at first restrict our considerations to the single change point scenario\footnote{However, as will be shown in the simulations, our findings do have practical implications on the multiple change point scenario as well which is why we stated the general model in \eqref{eq:basicmodel}.}.
\begin{assumption}\label{ass:amoc}
 We consider a single change point scenario with a change at some time point ~$u\in\{1,\ldots,n-1\}$ ~where ~$n\geq 3$.
\end{assumption}
 We need to clarify the selection of ~$\lambda$ ~for \eqref{eq:totvar}. In the single change point setup we aim to select ~$\lambda$ ~as large as possible 
 such that the set ~$\cE$ ~of change points contains only one change point\footnote{For the single change point scenario \citet[in their software GFLseg]{vert2011a} perform a dichotomic search to find the ``first'' ~$\lambda$ ~such that ~$\cE$ ~contains only one element.} in which case ~$\cE=\{\hat{u}\}$ ~and ~$\hat{u}$ ~denotes the {\it denoising estimate} for ~$u$. 
 Heuristically, this forces ~$\hat{u}$ ~to be the most reasonable selection of exactly one common change point according to the penalty term ~$\text{totvar}(U)$.

As will be discussed in Proposition \ref{prop:selection_procedure}, one can identify under mild assumptions 
a random interval such that any ~$\lambda\in (\lambda_{\min},\lambda_{\max})$, with ~$0\leq \lambda_{\min}<\lambda_{\max}$, 
yields the same estimate ~$\hat{u}$ ~and such that any ~$\lambda \in[\lambda_{\max},\infty)$ ~yields ~$\cE=\emptyset$, i.e. no estimate. 
Thus, in the following, we tacitly assume that we select any ~$\lambda\in(\lambda_{\min},\lambda_{\max})$ ~in which case the corresponding estimate ~$\hat{u}$ ~is unambiguous.

 Notice that generally the number of estimated change points does not necessarily decrease monotonously in ~$\lambda$ ~for ~$d>1$ ~(cf. Figure \ref{fig:nonmonotonic} and Section 4 of \cite{vert2011a}) and, 
 additionally, it seems not clear whether any parameter ~$\lambda$ ~that identifies only one change point yields the same estimate.   
 
 \begin{remark}\label{rem:selctionlambda}
 A parameter ~$\lambda$ ~that yields only one change point does not always exist. (E.g., it holds ~$\cE=\emptyset$ ~for any ~$\lambda$ ~if all entries of ~$Y$ ~are equal.) 
We will consider situations where such cases do not occur with probability tending to 1 as ~$d\rightarrow \infty$.  
 \end{remark} 
 
\begin{figure}
 \centering
	\includegraphics[width=0.4\textwidth, trim = 5mm 0mm 5mm 5mm]{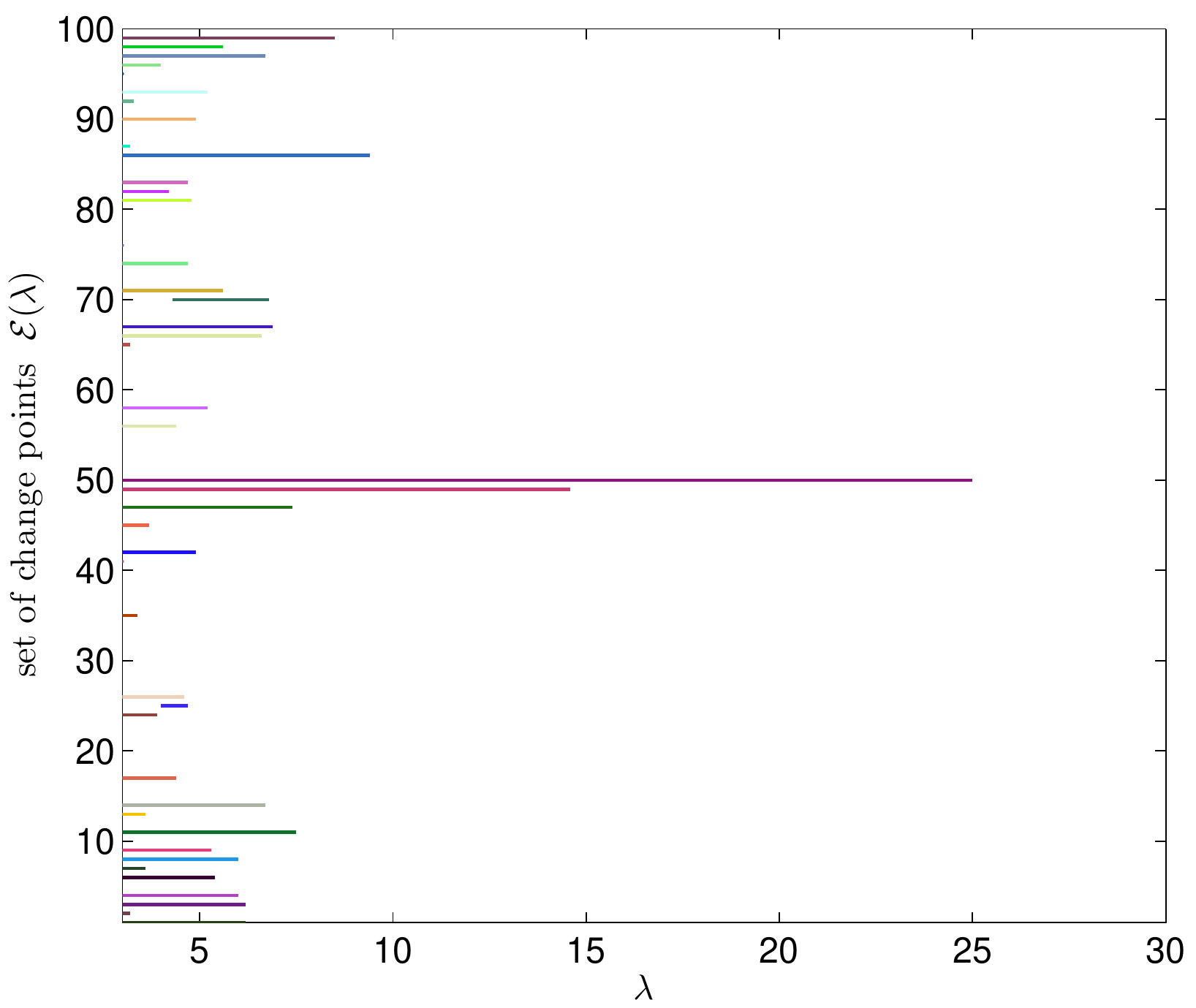}
\caption{The set ~$\cE(\lambda)$, as defined in \eqref{eq:set_changes}, along the ~$\lambda$ ~regularization path for ~$n=100$, $d=50$, independent panels 
with independent standard Gaussian noise ~$\{\varepsilon_{i,k}\}$ ~and without common factors, i.e. with ~$\gamma_k=0$; 
the true change point lies at ~$u=50$ ~with means ~$\mu_{1,k}=0$ ~and ~$\mu_{2,k}=1$, i.e. ~$\Delta_k=1$, ~for all panels ~$k=1,\ldots, d$. Notice that a change point at ~$i=25$ ~is 
only estimated for ~$\lambda$ ~between 4 and 4.7. (The computations are performed using the MATLAB package GFLseg.)}\label{fig:nonmonotonic}
\end{figure}

\begin{remark}
For selection of some reasonable ~$\lambda$ ~in the case of multiple changes, 
in particular if the number of changes is unknown in advance (which is a more realistic scenario), 
we refer to \citet{vert2011a} and to the references therein.
\end{remark}

\subsubsection{Relation to weighted CUSUM}\label{sec:denoising_cusum}
 
In the next proposition we observe, based on Proposition \ref{prop:selection_procedure}, the relation of the estimate ~$\hat{u}$ ~from the denoising approach to a well-known weighted CUSUM estimate\footnote{$\hat{u}_{\star}$ ~is usually defined as the smallest element in ~$\argmax$. Here, we allow ~$\hat{u}_{\star}$ ~to be any element in ~$\argmax$.}
\begin{equation}\label{eq:lasso_cusum}
	\hat{u}_{\star}\in\argmax_{i=1,\ldots,n-1}t(i), \qquad t(i)=w^2(i,n)\Bigg(\sum_{k=1}^d\big|\sum_{j=1}^i(Y_{j,k}-\bar{Y}_{n,k})\big|^2\Bigg).
\end{equation}

\begin{proposition}\label{prop:LASSO_CUSUM}
Under Assumption \ref{ass:amoc} and given that ~$t(i)$ ~has a unique maximum, it holds that ~$\hat{u}\equiv \hat{u}_{\star}$ ~if we use the same weighting ~$w$ ~for the 
denoising and the CUSUM estimates.
\end{proposition} 

Note that this connection holds true only in case of a single change point whereas otherwise 
denoising and CUSUM estimates differ. Therefore, recall that the denoising segmentation approach yields ~$p$ ~distinct change point estimates in case of ~$p$ ~change points.

Proposition \ref{prop:LASSO_CUSUM} allows us to study denoising and CUSUM estimates simultaneously if 
we tacitly exclude the case of non unique maxima of ~$t(i)$ ~from our considerations\footnote{If ~$t(i)$ ~has a non unique maximum, then counterexamples may be constructed 
such that ~$\cE(\lambda)=\{\hat{u}\}$ ~with ~$\hat{u}\in\argmax_{i=1,\ldots,n-1}t(i)$ ~is impossible.}. 
This is not a problem, since we will focus mostly on situations where this case does not occur with probability tending to 1 as ~$d\rightarrow \infty$.

\subsubsection{Common weighting schemes}\label{sec:common_schemes}

\citet{vert2011a} already studied the weightings
\begin{equation}\label{eq:standard_and_simple}
	w^\text{simple}(i,n)= 1,\qquad w^\text{standard}(i,n)=((i/n)(1-i/n))^{-1/2}
\end{equation}
for the denoising estimate ~$\hat{u}$ ~with respect to ~$d\rightarrow\infty$ ~and the latter has also been studied by \citet{Bai2010} for the least-squares estimate. 
Both schemes can be considered as natural and are reasonable for the denoising and for the CUSUM estimate as well. 
The former, ~$w^\text{simple}$, appears to be the first choice from the point of view of the denoising approach. 
In fact, \cite{vert2010} started with this case and studied ~$w^\text{standard}$ ~later in \cite{vert2011a}. 
On the other hand, the latter weighting,  ~$w^\text{standard}$, appears to be the natural choice from the CUSUM point of view, 
because it can be derived via a maximum-likelihood or a least-squares approach. Both weights are special cases of the following parametrized scheme
\begin{equation}\label{eq:admissible_functions}
	w^{\text{weighted}}(i,n)=((i/n)(1-i/n))^{-\gamma},\qquad 0\leq \gamma \leq 1/2. 
\end{equation}
These weights are quite popular in the field of change point analysis. 
Asymptotic properties are well studied for ~$n\rightarrow \infty$ ~for testing with weighted CUSUM 
statistics\footnote{When dealing with CUSUM (under ~$n\rightarrow\infty$ ~asymptotics), the observations ~$\{Y_{j,\bullet}\}$ ~are usually additionally rescaled by the long run covariance matrix.}, via ~$\max_{i=1,\ldots,n-1}t^{1/2}(i)$, or for estimating changes via ~$\argmax_{i=1,\ldots,n-1}t(i)$ (cf., e.g., \citet{horvath1997limit}). 
A smaller ~$\gamma$ ~is usually expected to increase the sensitivity of testing or estimation procedures towards change points in the middle of time series.

In the next subsection we will study estimates under the ~$d\rightarrow\infty$ ~asymptotics with respect to the weights \eqref{eq:standard_and_simple}  and  \eqref{eq:admissible_functions}. 
In particular we will see limitations of \eqref{eq:standard_and_simple} and propose a suitable extension ~$w^\text{exact}$ ~that has better consistency properties. 
 
\subsubsection{Theoretical analysis of weighting schemes}\label{sec:theoretical}

For our analysis we have to impose some (homogeneous) structure on the noise ~$\{\varepsilon_{i,k}\}$ ~and on the common factors ~$\{\zeta_{i}\}$ ~in the next two assumptions. Therefore, let 
\begin{equation}\label{eq:def_tdpartial}
	S_{i,k}(\varepsilon)=n^{-1/2}\sum_{j=1}^i (\varepsilon_{j,k} - \bar{\varepsilon}_{n,k})
\end{equation}
be the cumulated centered noises in the ~$k$-th panel.

\begin{assumption}\label{ass:common_structure} 
\leavevmode
\begin{enumerate} 
\item The noise ~$\{\varepsilon_{i,k}\}_{i,k\in\mN}$ ~is centered with finite fourth moments and the variances fulfill ~$E(\varepsilon_{i,k})^2=\sigma^2$ ~for some ~$0<\sigma^2<\infty$ ~and all ~$i,k$.
\item The function 
\[
	V^2(i)=\Var\left(S_{i,k}(\varepsilon)\right)/\sigma^{2}, \quad i=1,\ldots,n-1,
\]
is independent of ~$k$.
\end{enumerate}
\end{assumption}

\begin{assumption}\label{ass:common_factors}
The common factors ~$\{\zeta_{i}\}$ ~are independent of  ~$\{\varepsilon_{i,k}\}$ ~and are centered with finite fourth moments. Moreover, it holds that, as ~$d\rightarrow\infty$,
\begin{equation}\label{eq:condition_gammas}
	\frac{1}{d}\sum_{k=1}^d\gamma_k^2=o(1).
\end{equation}
\end{assumption}

Before proceeding further with the theory, we show some specific examples for the function ~$V^2(i)$ ~and also discuss some sufficient conditions for part 2 of Assumption \ref{ass:common_structure}. 
Clearly, given that part 1 of Assumption \ref{ass:common_structure} holds true, a sufficient condition is identical distribution of the panels ~$\{\varepsilon_{\bullet,k}\}_{k\in\mN}$. 
The following examples will both play important roles in our subsequent analysis.

 \begin{example}[Uncorrelated noise]\label{example:iid}
   Assume that part 1 of Assumption \ref{ass:common_structure} holds true and that ~$\{\varepsilon_{i,k}\}_{i\in\mN}$ ~are pairwise uncorrelated for any ~$k$. 
   In this situation it holds that
\begin{equation}\label{eq:weigthsstandardiid}
	V^2(i)=(i/n)(1-i/n).
\end{equation} 

\end{example}
 \begin{example}[Moving average noise]\label{example:MA1}
Another interesting case, which satisfies Assumption \ref{ass:common_structure}, is given by ~$\{\varepsilon_{i,k}\}_{i,k\in\mN}$ ~where
\begin{equation}\label{eq:def_ma1}
	\varepsilon_{i,k} = \left(\eta_{i,k}+\phi \eta_{i-1,k}\right)+\theta \left(\eta_{i,k-1}+\phi \eta_{i-1,k-1}\right)
\end{equation}
for ~$i,k\in\mN$, i.e. ~$\{\varepsilon_{i,k}\}_{i,k\in\mN}$ ~are  MA(1) in time and across panels. 
Here, we assume some common parameters ~$\phi,\theta\in\mR$ ~and centered i.i.d. shocks ~$\{\eta_{i,k}\}_{i,k\in\mN}$ ~with finite 
fourth moments and with ~$E(\eta_{i,k}^2)=\tilde{\sigma}^2$, ~$0<\tilde{\sigma}^2<\infty$. In this case \eqref{eq:weigthsstandardiid} extends to
\begin{equation}\label{eq:MA1cov}
	V^2(i)=C\alpha(\phi)\Big[(i/n)(1-i/n)\Big]-2C\phi/n
\end{equation}
with 
\begin{equation}\label{eq:alpha_and_sigma}
	\alpha(\phi)=1+\phi^2+2\phi+2\phi/n,\qquad \sigma^2=\tilde{\sigma}^2(1+\phi^2+\theta^2+\phi^2\theta^2)
\end{equation} 
and with the constant ~$C=(\tilde{\sigma}/\sigma)^2(1+\theta^2)$ ~that is independent of ~$i$. 
\end{example}

The following deterministic {\it critical functions} are the cornerstone of our subsequent analysis:
\begin{equation}\label{eq:def_critical_function}
	C_n(i;u,r) = w^2(i,n) \Big [ V^2(i)r +  H^2(i,u) \Big]
\end{equation} 
for ~$0< r <\infty$ ~with
\[
	H(i,u)=
\begin{cases}
(i/n)(1-u/n),& i=1,\ldots,u,\\
(u/n)(1-i/n),& i=u,\ldots,n-1,
\end{cases}
\] 
for ~$u=1,\ldots,n-1$. As before, ~$u=u_1$ ~is the change point and ~$w(i,n)$ ~are the weights. The parameter
\begin{align}\label{eq:noisechangeratio}
	r&=\bigg(\frac{1}{n}\frac{\sigma^2}{\Delta^2_\infty}\bigg),\\
\intertext{for ~$0< \sigma^2<\infty$ ~and ~$0<\Delta^2_\infty< \infty$, is the normalized {\it noise to change ratio} where}
	\Delta^2_\infty&= \lim_{d\rightarrow \infty}\bigg(\frac{1}{d}\sum_{k=1}^d\Delta_k^2\bigg)
\end{align}
is the average change across all panels at the change point ~$u$. (Recall from \eqref{eq:basicmodel} and \eqref{eq:def_delta} that  ~$\Delta_k=\mu_{2,k}-\mu_{1,k}$.) The case ~$\Delta^2_\infty=0$ ~of a {\it vanishing change} is excluded but will be addressed in the Remark \ref{rem:nochange} below. For simplicity we will write ~$C(i;u,n,r)=C_n(i;u,r)$. 
\\
\\ 
The next theorem generalizes\footnote{Under Gaussianity and independence within panels the results coincide up to a normalizing constant.} \citet[Lemma 1]{vert2011a}. 
Following their approach we show that ~$C(i;u,n,r)$ ~is in probability the limit of rescaled ~$t(i)$ ~for all ~$i=1,\ldots,n-1$ ~(cf. also Proposition \ref{prop:selection_procedure}).

\begin{theorem}\label{thm:convergence_2max} Let Assumptions \ref{ass:amoc}, \ref{ass:common_structure} and \ref{ass:common_factors} be fulfilled. Assume that it holds that, as ~$d\rightarrow\infty$,  
\begin{align}
\frac{1}{d^2}&\sum_{k,q=1}^d \big|\Cov (\varepsilon_{j,k}\varepsilon_{h,k},\varepsilon_{l,q}\varepsilon_{p,q})\big|=o(1),\label{eq:condition2}\\
\frac{1}{d^2}&\sum_{k,q=1}^d\alpha_{k,q}\big|\Cov (\varepsilon_{j,k},\varepsilon_{l,q})\big|=o(1)\label{eq:condition1}
\end{align}
with ~$\alpha_{k,q}=|\Delta_k\Delta_q| +|\gamma_k\gamma_q|$ ~and for all ~$j,h,l,p=1,\ldots,n$. Then, it holds that, as ~$d\rightarrow\infty$,
\begin{equation}\label{eq:perfect_estimation_not_the_definition}
	P\left({\argmax_{i=1,\ldots,n-1}} t(i) \subseteq S\right)\rightarrow 1, \qquad S={\argmax_{i=1,\ldots,n-1}} C(i;u,n,r).
\end{equation}
\end{theorem}  
 
Theorem \ref{thm:convergence_2max} immediately implies ~$P(\hat{u}_{\star}\in S)\rightarrow 1$. 
Note that, in view of Proposition \ref{prop:LASSO_CUSUM}, the same limiting behaviour holds true for the denoising estimate ~$\hat{u}$ ~if 
the maximum of ~$C(i;u,n,r)$ ~is unique - which will be the interesting case in this article.

\begin{remark}\label{rem:conditionsforexamples} Let Assumptions \ref{ass:amoc}, \ref{ass:common_structure} and \ref{ass:common_factors} hold true.
 Clearly, conditions \eqref{eq:condition2} and \eqref{eq:condition1} are fulfilled if ~$\{\varepsilon_{i,k}\}$ ~are i.i.d. but deviations, 
 in particular within panels, are also possible. For example if panels ~$\{\varepsilon_{\bullet,k}\}$ ~are independent then
 \begin{equation}\label{eq:cond_supvar}
 	\sup_{j,h=1,\ldots,n,\;k\geq 1}\Var(\varepsilon_{j,k}\varepsilon_{h,k})<\infty
\end{equation}
is sufficient. Condition \eqref{eq:cond_supvar} may be reduced further to ~$\Var(\varepsilon_{j,1}\varepsilon_{h,1})<\infty$ ~for all ~$j,h=1,\ldots,n$ ~if we 
additionally assume identical distribution of ~$\{\varepsilon_{\bullet,k}\}$. Notice that it is straightforward to check that ~$\{\varepsilon_{\bullet,k}\}$ ~from 
Example \ref{example:MA1} fulfill conditions \eqref{eq:condition2} and \eqref{eq:condition1}, too.  
\end{remark}
 
%\newpage
\paragraph{Perfect estimation and the exact weighting scheme}\hfill\\\\    
Given any estimate ~$\tilde{u}$ ~for a change point ~$u$, we will speak of {\it perfect estimation}\footnote{Note that perfect estimation was previously defined in \citet{torg2015v2} in terms of the limiting critical function ~$C$.} if consistent estimation, i.e. 
\begin{equation}\label{eq:cons_est_def}
	P(\tilde{u}=u)\rightarrow1,
\end{equation}
as ~$d\rightarrow\infty$,	holds true for all possible change points ~$u=1,\ldots,n-1$ ~and all possible ratios ~$0< r<\infty$. 
 
Theorem \ref{thm:convergence_2max} shows that under mild assumptions the stochastic limits of CUSUM and denoising estimates are described by the deterministic critical function \eqref{eq:def_critical_function}. This reduces the question of consistency and of perfect estimation to an analytical problem: Given \eqref{eq:perfect_estimation_not_the_definition} it is sufficient for consistency if the following Assumption A1 holds true. Clearly, the same applies to perfect estimation if we require the assumption to hold for all possible change points ~$u=1,\ldots,n-1$ ~and all possible ratios ~$0< r<\infty$.  

\begin{assA}\label{ass:A}
The critical function ~$C(i;u,n,r)$ ~has a unique maximum at ~$i=u$ ~given a change point at ~$u$.
\end{assA}
We proceed by studying the existence of weights which ensure perfect estimation for the denoising and the CUSUM estimates. The weights are tacitly assumed not to depend on ~$u$.
  
\begin{theorem}\label{thm:charct_eq} Let Assumption \ref{ass:amoc} be fulfilled and assume that ~$V$ ~is strictly positive. Only the weights
\begin{equation}\label{eq:charct_eq}
	w^{\text{exact}}(i,n)= \alpha/V(i), \qquad \alpha>0,
\end{equation}
for ~$i=1,\ldots,n-1$ ~may fulfill Assumption A1 for all possible change points ~$u=1,\ldots,n-1$ ~and all ratios ~$0< r<\infty$.  
For weights other than \eqref{eq:charct_eq} there is some change point ~$u$ ~and some ratio ~$r>0$ ~such that the maximum of the critical function ~$C(i;u,n,r)$ ~is not at ~$i=u$. 
\end{theorem}
Since the estimates are independent of any scaling ~$\alpha>0$, 
we may restrict ourselves to the case of ~$\alpha=1$ ~and consider ~$w^{\text{exact}}$ ~to be unique.  
Notice that in the setting of Example \ref{example:iid} the schemes ~$w^{\text{exact}}$ ~and ~$w^{\text{standard}}$ ~coincide and 
that we already know from \citet[Theorem 3]{vert2011a} that these weights yield perfect estimation for our estimates in the Gaussian i.i.d. setting. 
However, as follows from Example \ref{example:MA1}, we have generally ~$w^{\text{exact}}\neq w^{\text{standard}}$ ~if the noise is dependent in time. 
Hence, due to Theorem \ref{thm:charct_eq}, weights ~$w^{\text{standard}}$ ~cannot generally ensure consistency and perfect estimation in 
such cases\footnote{CUSUM (and denoising) estimates are not consistent if the (unique) maximum of ~$C(i;u,n,r)$ ~is not at ~$i=u$.}. 
Notice that the weightings ~$w^\text{standard}$ ~and ~$w^\text{exact}$ ~might differ fundamentally. The former is strictly convex but the latter is even 
strictly concave in case of \eqref{eq:MA1cov} for ~$\alpha(\phi)<0$.

A consequence of the next theorem, under the assumptions of Theorem \ref{thm:convergence_2max}, is that covariance-based weights $w^{\text{exact}}$ ensure perfect estimation under additional (but again) quite mild assumptions on $V(i)$. \!Therefore, we need to find conditions such that $H(i,u) / V(i)$ has a (unique) maximum at ~$i=u$ ~for any ~$u=1,\ldots,n-1$.

\begin{theorem}\label{thm:cond_v} Let Assumption \ref{ass:amoc} be fulfilled. Assume that we use the weights ~$w^{\text{exact}}$, that ~$V$ ~is strictly positive and that ~$V(i)=V(n-i)$ ~holds true for all ~$i$. Then Assumption A1 is fulfilled for all possible change points ~$u=1,\ldots,n-1$ ~and all ratios ~$0< r<\infty$ ~if and only if
\begin{equation}\label{eq:cond_V_sufficient}
	V(i)/i>V(u)/u,\quad i,u=1,\ldots,n-1,\quad i<u
\end{equation}
holds true. This condition is equivalent to ~$V(i)/i$ ~being strictly decreasing. 
\end{theorem}

Note that the symmetry ~$V(i)=V(n-i)$ ~for all ~$i=1,\ldots,n-1$ ~is implied by the symmetry of covariances
\begin{equation}\label{eq:covstat}
	\Cov((\varepsilon_{1,1},\ldots,\varepsilon_{n,1})^T) = \Cov((\varepsilon_{n,1},\ldots,\varepsilon_{1,1})^T)
\end{equation}
and a sufficient condition for \eqref{eq:covstat} is weak stationarity of ~$\{\varepsilon_{i,1}\}$.  The next lemma provides, based on concavity, a condition for \eqref{eq:cond_V_sufficient} which is sometimes easier to verify than monotonicity of ~$V(i)/i$ ~and that will also be used in Remark \ref{rem:verify_ma1}. It is not clear how to state a comparable condition under convexity.
\begin{lem}\label{lem:concave_func}
Assume that ~$V(i)$ ~is strictly concave\footnote{As a discrete function linearly interpolated on the interval ~$[1,n-1]$.}, strictly positive and that ~$n\geq 3$. Then \eqref{eq:cond_V_sufficient} already holds true if
\begin{equation}\label{eq:criterion_simple}
	V(1)>V(u)/u
\end{equation}
holds true for all ~$u=2,\ldots,n-1$.
\end{lem}

\begin{remark}\label{rem:verify_ma1} 
The function ~$V(i)$ ~from Example \ref{example:MA1} fulfills \eqref{eq:cond_V_sufficient} with any moving average parameters ~$\phi,\theta\in\mR$ ~and is always strictly positive for $n\geq 3$.
\end{remark}

Now, a combination of Theorems \ref{thm:convergence_2max} and \ref{thm:cond_v}  together with Remarks \ref{rem:conditionsforexamples}  and \ref{rem:verify_ma1} yield the following corollary.

\begin{corollary} 
Consider the moving average noise from Example \ref{example:MA1} and let Assumptions \ref{ass:amoc} and \ref{ass:common_factors} hold true. Using ~$w^{\text{exact}}$ ~we have perfect estimation for denoising and CUSUM estimates for any moving average parameters ~$\phi,\theta\in\mR$.
\end{corollary}

\paragraph{Consistent estimation and the weighted weighting scheme}\hfill\\\\
We turn to the analysis of ~$w^{\text{weighted}}$ ~and, analogously to \citet[Theorem 3]{pesta2015}, our aim is to identify noise to change ratios ~$r$ ~for which consistent estimation \eqref{eq:cons_est_def} does or does not hold true for the denoising and the CUSUM estimates. We restrict the consideration to panels which fulfill all assumptions of Theorem \ref{thm:convergence_2max}  with ~$V^2(i)=(i/n)(1-i/n)$ ~from Example \ref{example:iid} (cf. Remark \ref{rem:conditionsforexamples}). We expect more restrictive conditions on ratios ~$r$ ~for changes ~$u$ ~closer to the edges of the samples, and vice versa less restrictive conditions if ~$u$ ~is more centered. This expectation is confirmed by the next theorem where the minimum is taken over smaller sets in \eqref{eq:minbound} if ~$u$ ~is closer to ~$n/2$.

\begin{theorem}\label{thm:locationmaxiid} Let Assumption \ref{ass:amoc} be fulfilled. Assume ~$V(i)$ ~as in \eqref{eq:weigthsstandardiid} and assume that we use ~$w^{\text{weighted}}$ ~with ~$\gamma\in[0,1/2)$. Define ~$R(u,i)$ ~as in \eqref{eq:ratios} below and set
\begin{equation}\label{eq:minbound}
	\cR:=\min_{n/2\leq i<u^*}R(u^*,i)>0
\end{equation}
with
\[
u^* =
\begin{cases}
	 n-u,  &u=1,\ldots, \lfloor n/2\rfloor, \\
	 u,  &u=\lceil n/2 \rceil,\ldots, n-1\\
\end{cases}
\] 
and where ~$\min \emptyset = \infty$. Then Assumption A1 holds true for ~$0\leq r<\cR$ ~and does not hold true for ~$\cR<r$. In the latter case the maximum of the critical function ~$C(i;u,n,r)$ ~is not at ~$i=u$. 
%Furthermore, ~$\cR =\infty$ ~if ~$n=3$ ~or if ~$u^*=\lceil n/2\rceil$.
\end{theorem}

The bound in \eqref{eq:minbound} can be evaluated numerically but to gain more insight into the influence and the interaction of parameters, it is desirable to get explicit representations and approximations of this expression. We will provide such approximations where we first let ~$d\rightarrow\infty$ ~and then consider ~$n\rightarrow\infty$. 
Therefore, we have to introduce a {\it boundary function}
\begin{equation}\label{eq:boundary}
	B(\gamma) = \begin{cases}\frac{4\gamma^2 +6\gamma^{3/2}-3\gamma^{1/2}-1}{8\gamma^2 + 8\gamma^{3/2} -4\gamma^{1/2}-2\gamma-1},&\gamma\in[0,1/2),\\
2^{-1/2},&\gamma=1/2.
\end{cases}
\end{equation}
This function ~$B(\gamma)$ ~is monotonously decreasing, continuous and ~$B(\gamma)\geq 2^{-1/2}$ ~holds true which can be seen as follows. It holds that ~$\partial_\gamma (B(\gamma^2))=-(1+2\gamma)^{-2}$ ~which in turn implies that ~$\partial_\gamma B(\gamma)<0$ ~holds true for any ~$\gamma\in(0,1/2)$. Applying l'H\^{o}pital's rule to ~$B(\gamma^2)$ ~we get the continuity at ~$\gamma=1/2$.

\begin{theorem}\label{thm:representmaxiid}
Let the assumptions of Theorem \ref{thm:locationmaxiid} hold true, set ~$B(\gamma)$ ~as in \eqref{eq:boundary}, set ~$s=u^*/n$ ~and let ~$\cR=\cR(s,\gamma)$ ~be as in \eqref{eq:minbound}.  
\begin{enumerate}
\item If ~$1/2 +1/n<s\leq B(\gamma)$, then ~$\cR(s,\gamma)$ ~equals to the unique solution of 
\begin{equation}\label{eq:slopecriteria}
	C(u^*-1;u^*,n,r)=C(u^*;u^*,n,r)
\end{equation}
and for ~$u^*=\lfloor n\zeta \rfloor$ ~with any ~$1/2<\zeta\leq B(\gamma)$ ~it holds that
\begin{equation}\label{eq:asymptotic_expression_gam}
	\lim_{n\rightarrow\infty}\cR(s,\gamma)=\zeta(a-\zeta)(b-\zeta)(\zeta-c)^{-1}
\end{equation}
with ~$a=1$, $b=(\gamma-1)/(2\gamma-1)$, $c=1/2$. 
\item If  ~$B(\gamma)+2/n<s< 1$, then the unique solution to \eqref{eq:slopecriteria} is larger than ~$\cR(s,\gamma)$.
\end{enumerate}
\end{theorem}

In \eqref{eq:asymptotic_expression_gam} the quantity ~$\cR(s,\gamma)$, i.e. the range for ~$r$ ~that ensures consistency, becomes larger for parameters ~$\gamma$ ~and ~$\zeta$ ~close to ~$1/2$, which confirms our intuition.  Note that Theorem \ref{thm:representmaxiid} is consistent with \citet[Theorem 2]{vert2011a} since, as ~$n\rightarrow\infty$, it holds
\[
	\cR(s,0)=\zeta(1-\zeta)^2(\zeta-1/2)^{-1}+o(1)
\] and ~$B(0)=1$. 
\\
\\
The weighting with ~$\gamma=1/4$ ~can be seen as a compromise between ~$w^{\text{standard}}$ ~and ~$w^{\text{simple}}$. For this particular  ~$\gamma$ ~we are able to compute ~$\lim_{n\rightarrow\infty}\cR(s,1/4)$ ~for any ~$\zeta\in(1/2,1)$. It would be interesting to know if such a formula could be computed for the remaining parameters  ~$\gamma\in(0,1/4)\cup(1/4,1/2)$ ~as well.

\begin{proposition}\label{prop:gamma025}
Let the assumptions of Theorem \ref{thm:representmaxiid} hold true, set ~$\gamma=1/4$ ~and ~$B(\gamma)$ ~as in \eqref{eq:boundary}, i.e. ~$B(1/4)=3/4$. Then it holds for ~$\cR=\cR(s,1/4)$ ~that
\begin{equation}\label{eq:limitfunc}
\lim_{n\rightarrow \infty} \cR(s,1/4) =
\begin{cases}
\frac{\zeta(1-\zeta)(3/2-\zeta)}{(\zeta-1/2)},& \zeta\in(1/2,3/4],\\
\frac{(\zeta-1)^2(\zeta(\zeta+1)+(\zeta/2-1)(1-\zeta)^{1/2}-1)}{\zeta(\zeta-1)+(\zeta/2)(1-\zeta)^{1/2}},&\zeta\in(3/4,1)
\end{cases}
\end{equation}
This function is continuously differentiable for ~$\zeta\in(1/2,1)$. The second derivative does however not exist for ~$\zeta=3/4$.
\end{proposition}

\paragraph{Spurious estimation for vanishing change points}\hfill\\\\ 
We close this subsection by a short remark on spurious estimation for ~$\Delta_\infty=0$, i.e. a probably common change that vanishes asymptotically. We also include the case ~$\Delta_k=0$ ~for ~$k\geq 1$. 
\begin{remark}\label{rem:nochange}
Let all assumptions of Theorem \ref{thm:convergence_2max} hold true but with ~$\Delta_\infty=0$.  Following the proof of Theorem \ref{thm:convergence_2max} it is clear that \eqref{eq:perfect_estimation_not_the_definition} also holds true in this situation with ~$C(i;u,n,r) := 
w^2(i,n) V^2(i)$. In case of Example \ref{example:iid} and for ~$w^{\text{weighted}}$ ~this yields
\[
	\argmax_{i=1,\ldots,n-1} C(i;u,n,r)= 
	\begin{cases} 
	\{\lfloor n/2 \rfloor,\lceil n/2 \rceil\},& \gamma\in[0,1/2),\\
   \{1,\ldots,n-1\}, & \gamma=1/2,
	\end{cases} 
\]
i.e. estimation of spurious changes. In case of Example \ref{example:MA1} and using ~$w^{\text{exact}}$ ~it also always holds that ~$\argmax_{i=1,\ldots,n-1} C(i;u,n,r)=\{1,\ldots,n-1\}$ ~since ~$C(i;u,n,r)$ ~is constant in ~$i$.
\end{remark} 
\subsubsection{Estimation of the exact weighting scheme}\label{sec:estimation} 
In this subsection we discuss estimates of ~$w^{\text{exact}}(i,n)$, or equivalently of ~$V(i)$, under the assumptions that ~$V$ ~is strictly positive, that the sequences ~$\{\varepsilon_{j,k}\}_{k\in\mN}$ ~and ~$\{\varepsilon_{j,k}\varepsilon_{l,k}\}_{k\in\mN}$ ~fulfill the weak law of large numbers for any ~$1\leq j,l\leq n$ ~and that Assumptions \ref{ass:amoc}, \ref{ass:common_structure} and \ref{ass:common_factors} hold true. Notice that all conditions on the noise ~$\{\varepsilon_{j,k}\}$ ~are clearly fulfilled for our moving average Example \ref{example:MA1}.
\\
\\
The partial sums \eqref{eq:def_tdpartial} can be rewritten as
\begin{equation}\label{eq:partial_sum_rewritten}
	S_{i,k}(\varepsilon) = \sum_{j=1}^n a_{i,j}\varepsilon_{j,k}
\end{equation} 
with
\begin{equation}\label{eq:def_aij}
a_{i,j} =\frac{1}{n^{1/2}}
	\begin{cases}
1-i/n,& j=1,\ldots,i,\\
-i/n,& j=i+1,\ldots,n
	\end{cases}
\end{equation}
and it is straightforward to check that
\begin{equation}\label{eq:est_with_sigma}
	V^{2}(i)=f_i(\Sigma)/\sigma^{2},
\end{equation}
where ~$\Sigma=\Cov(Y_{\bullet,1})$, with functions ~$f_i(\Sigma)=a_{i,\bullet}\Sigma a_{i,\bullet}^T$ ~for ~$i=1,\ldots,n-1$.
\\
\\
A natural estimate for ~$\Sigma$ ~is given via ~$\hat{\Sigma}$ ~with
\begin{equation}\label{eq:natural_estimate}
	\hat{\Sigma}_{j,k}=\frac{1}{d-1}\sum_{p=1}^{d} (Y_{j,p} - \bar{Y}_{j,d})(Y_{k,p} - \bar{Y}_{k,d}),
\end{equation}
where ~$1\leq j,k\leq n$ ~and ~$d>1$. This estimate ~$\hat{\Sigma}_{j,k}$ ~is consistent\footnote{The calculations are straightforward and the contribution of common factors vanishes asymptotically due to Assumption \ref{ass:common_factors}.}, as ~$d\rightarrow \infty$, e.g. if we assume additionally that ~$\cM_{i,k}=c_i$ ~always holds true for some ~$c_i\in\mR$ ~and all ~$k\geq 1$, i.e. that at all time points the means across all panels are the same.

Now, ~$V^{2}(i)$ ~simply may be estimated by ~$\hat{V}^{2}(i)=f_i(\hat{\Sigma})/\sigma^{2}$ ~and the corresponding estimate for the weights will be denoted by ~$\hat{w}^{\text{exact}}$. Here,  ~$\bar{Y}_{j,d}$ ~is the mean over all panels at time point ~$j$. Recall also, that our change point estimates do not depend on the scaling of the weights ~$w^{\text{exact}}(i,n)$. Thus, without loss of generality we may assume here that ~$\sigma^2=1$ ~and technically we do not have to estimate this parameter in \eqref{eq:est_with_sigma}.
 
\begin{remark}\label{rem:problematicestimation} Reasonable estimates for the exact weights have to be strictly positive which corresponds to positiveness of ~$f_i(\hat{\Sigma})$. This is ensured asymptotically with probability tending to 1, as ~$d\rightarrow\infty$, because the estimate ~$\hat{\Sigma}$ ~is consistent and because of our usual assumption ~$V(i)>0$. (Clearly, a sufficient condition for finite ~$d$ ~would be the positive definiteness of the estimate ~$\hat{\Sigma}$.)
\end{remark} 

\paragraph{Banded estimation based on a training period}\hfill\\\\ 
Now, we aim to increase the precision of the estimate for ~$\Sigma$ ~by averaging and by a banded covariance approach. To do so we have to impose some structural assumptions: We assume weak stationarity for ~$\{\varepsilon_{j,k}\}_{j\in\mN}$ ~to hold in time, i.e. within each panel and additionally, we assume to have a {\it training period} between ~$n_1$ ~and ~$n_2$,  ~$1\leq n_1< n_2\leq n$ ~where the above assumptions of ~$\cM_{i,k}=c_i$ ~hold true for ~$i=n_1,\ldots,n_2$ ~and  ~$k\geq 1$. Within this training period we can compute consistent estimates ~$\hat{\Sigma}^\prime_{j,k}\approx \Cov(\varepsilon_{j,1},\varepsilon_{k,1})$, according to \eqref{eq:natural_estimate}, for ~$j,k=n_1,\ldots,n_2$ ~and  then may average these estimates as follows
\begin{equation}\label{eq:estimate_centered1}
\xi_{r}=\frac{1}{n_2-n_1-r+1}\sum_{j=n_1}^{n_2-r} \hat{\Sigma}^\prime_{j,j+r}\qquad \Big(\approx \Cov(\varepsilon_{1,1},\varepsilon_{1+r,1})\Big)
\end{equation}
for ~$r=0,\ldots,n_2-n_1$ ~to gain more stability. (These estimates are consistent as well.)  Finally, the desired banded estimate ~$\hat{\Sigma}$ ~is obtained via
\begin{equation}\label{eq:estimate_centered2}
	\hat{\Sigma}_{j,j+r}=\hat{\Sigma}_{j+r,j}=
	\begin{cases}
	\xi_{r},& r=0,\ldots,\min\{h,n-j\},\\
	0,&  r=\min\{h,n-j\}+1,\ldots,n
	\end{cases}
\end{equation}
for  ~$j=1,\ldots,n$ ~and with some banding parameter ~$h\in\{0,\ldots,n_2-n_1\}$ ~to be chosen.   Clearly, the corresponding estimate ~$\hat{V}^{2}(i)=f_i(\hat{\Sigma})/\sigma^{2}$ ~is consistent, as ~$d\rightarrow\infty$, for panels that are MA($q$) in time if ~$q\leq h\leq  n_2-n_1$ ~holds true. Heuristically, it yields also a reasonable approximation for ~$\hat{V}^{2}(i)=f_i(\hat{\Sigma})/\sigma^{2}$ ~in case of stationary and weak dependent panels, again in the time-domain, whenever the covariances ~$\Cov(\varepsilon_{1,1},\varepsilon_{1+r,1})$ ~for ~$r>h$ ~are negligible. The corresponding estimate for the weights will be denoted by ~$\hat{w}^{\text{exact-banded}}$. 
For a data-driven selection of the banding parameter ~$h$ ~one could use the approach of \citet[Section 5 and disp. (24)]{bickel2008} (cf., e.g., also \cite{wu2009}). In their simulations for MA(1) covariance structure, the banding parameter is always chosen correctly, i.e. ~$h=1$.

The assumptions stated above are restrictive and may be questionable in applications. In the following we (informally) discuss estimates for more complex situations when ~$\cM_{i,k}=c_i$ ~for ~$k\geq 1$ ~does not hold true within a reasonable subsample. 
Again, we need to assume a training period between ~$n_1$ ~and ~$n_2$,  ~$1\leq n_1< n_2\leq n$, such that either ~$n_2\leq u$ ~or ~$n_1>u$ ~holds true, i.e. that a common change does not occur in this subsample.  Further, we need to assume stationarity and the weak law of large numbers to hold in time for ~$\{\varepsilon_{j,k}\}_{k\in\mN}$, i.e. within each panel.  
Now, in the first step, we center each panel based on means computed within the training period and for each panel separately, i.e.
\[
	Y_{i,k}':=Y_{i,k} - \frac{1}{n_2-n_1+1}\sum_{j=n_1}^{n_2}Y_{j,k}
\]
for ~$i=1,\ldots,n$ ~and ~$k=1,\ldots,d$. In the next step (as before) we compute only estimates ~$\hat{\Sigma}^\prime_{j,k}$ ~for ~$n_1\leq j \leq k \leq n_2$ ~but now based on the centered panels ~$\{Y_{j,k}'\}$. Proceeding as under \eqref{eq:estimate_centered1} and \eqref{eq:estimate_centered2} we obtain ~$\hat{V}^{2}(i)=f_i(\hat{\Sigma})/\sigma^{2}$ ~and the corresponding weights ~$\hat{w}^{\text{exact-banded}}_{\text{centered}}$. Those are heuristically reasonable for large ~$n$, ~$d$ ~and a relatively large training period which is backed up by our simulations. However, to formalize this one would need to consider asymptotics ~$n,d\rightarrow\infty$ ~with ~$|n_2-n_1|\rightarrow\infty$ ~which is not in the scope of this paper.

%\newpage
\section{Simulations}\label{sec:simulations} 
For our simulations within the single change point scenario we have implemented the estimates \eqref{eq:lasso_cusum}\footnote{We choose ~$\hat{u}_{\star}$ ~as the smallest element in ~$\argmax t(i)$.} in MATLAB. For demonstration purposes an application with a graphical user interface can be obtained from the author or from \url{www.mi.uni-koeln.de/~ltorgovi}. For the simulations within the multiple change point scenario we work with the MATLAB ``GFLseg''-package of \cite{vert2011b}. Notice that the denoising approach may be interpreted as a group fused LASSO (least absolute shrinkage and selection operator) (cf. Section \ref{sec:proofs}) and that the corresponding group fused LARS (least angle regression) method yields a fast approximation to the LASSO solution (cf. \cite{vert2011a}). In particular, we use the LASSO and LARS methods that are implemented in gflassoK.m or in gflars.m respectively.

We proceed by considering panels ~$\{Y_{i,k}\}$ ~with the following parameters: The noise ~$\{\varepsilon_{i,k}\}$ ~is moving average from Example \ref{example:MA1} based on independent Gaussian innovations ~$\{\eta_{i,k}\}$. The common factors ~$\{\zeta_i\}$ ~are chosen to be independently uniformly distributed, centered and with the same variance as the noise ~$\{\eta_{i,k}\}$. Unless stated otherwise, the factor loadings are set to ~$\gamma_k=k^{-1/2}$ ~and the length of panels is ~$n=100$. For simplicity, we choose common changes  with ~$\mu_{1,k}=0$ ~and ~$\mu_{2,k}=1$ ~for all ~$k$, i.e. ~$\Delta_k=1$ ~and therefore ~$\Delta_\infty=1$. For ~$\hat{w}^{\text{exact-banded}}$~and for ~$\hat{w}^{\text{exact-banded}}_{\text{centered}}$ ~the training period is chosen as ~$n_1=1$ ~and ~$n_2=20$ ~with a  bandwidth ~$h=2$. (The influence of the misspecification of ~$h$ ~is rather mild in our settings.)

  All Monte Carlo simulations will be based on 100 repetitions. Notice that ~$w^{\text{standard}}$ ~replaces ~$\hat{w}$, whenever the corresponding estimate ~$\hat{V}=1/\hat{w}$ ~has at least one non-positive entry (cf. Remark \ref{rem:problematicestimation}).

\begin{figure}%
\centering
\subfloat[][Rescaled ~$\tilde{t}(i)$'s]{% 
\includegraphics[width=0.40\textwidth, trim = 5mm 0mm 5mm 5mm]{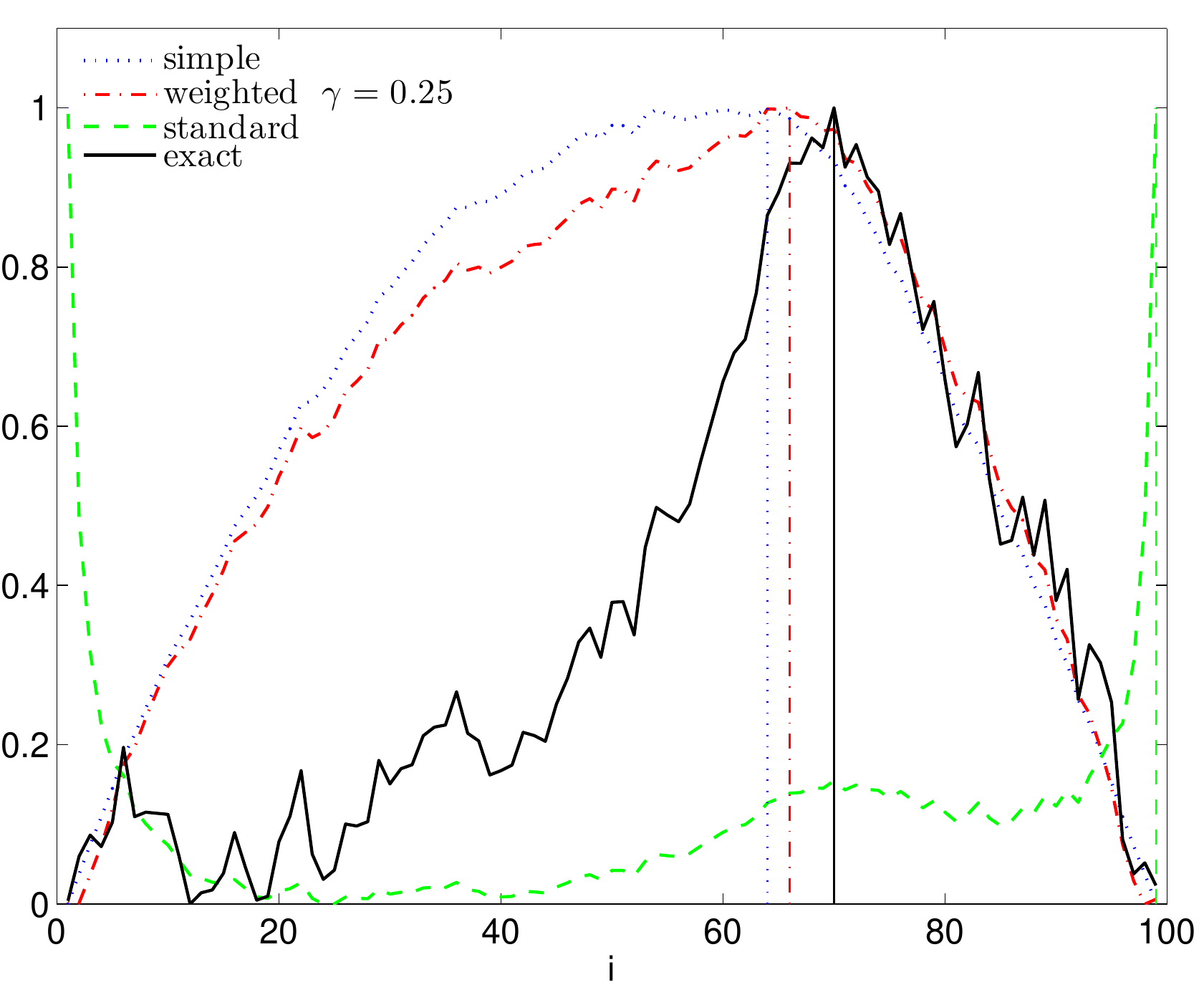}
}%
\hspace{40pt}%
\subfloat[][Rescaled critical curves ~$\tilde{C}_n(i;u,r)$]{% 
\includegraphics[width=0.40\textwidth, trim = 5mm 0mm 5mm 5mm]{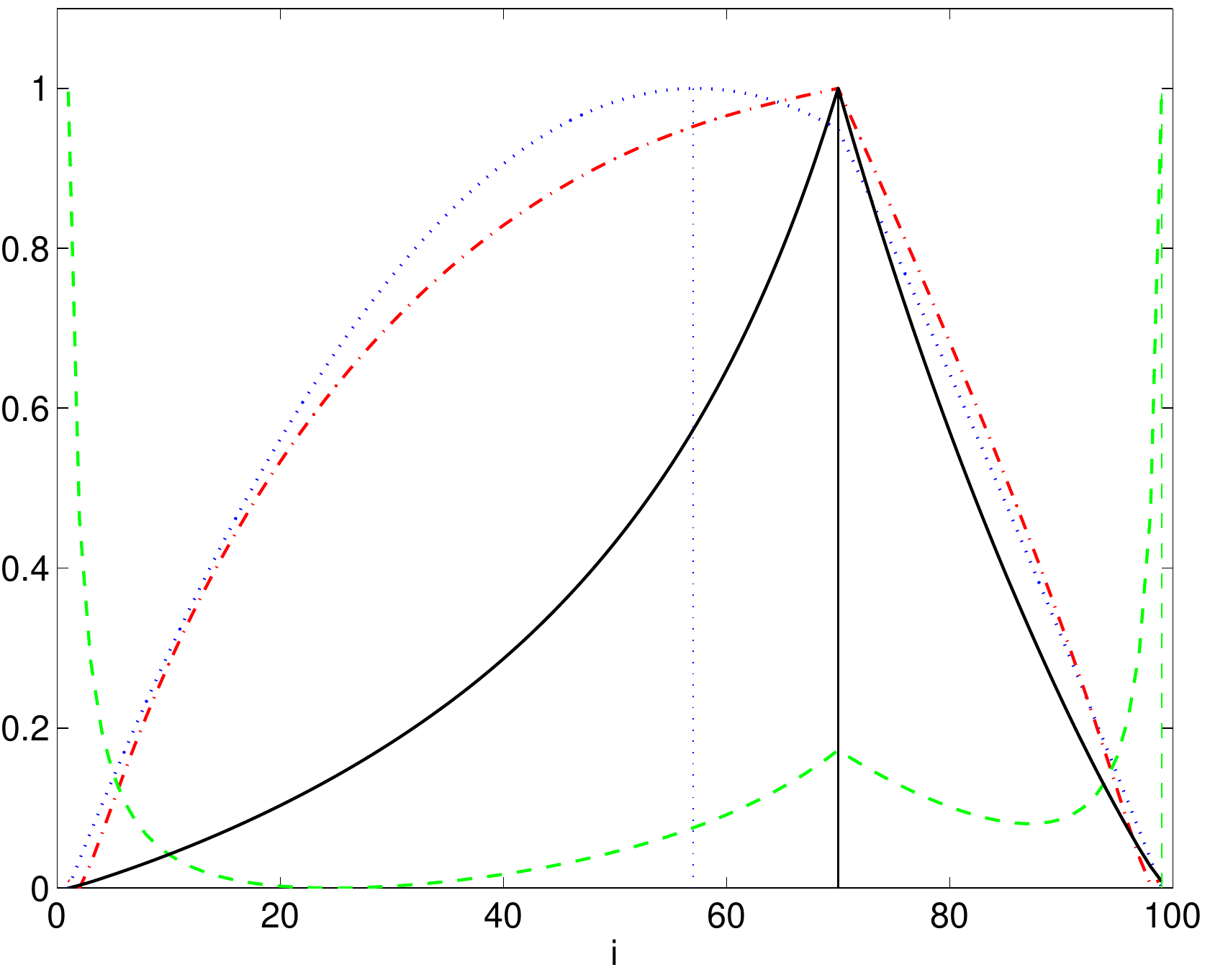}
}%
\caption{Estimation of the change point at ~$u=70$ ~with different weighting schemes. The parameters are ~$\tilde{\sigma}^2=9$, $\phi=-3$, $\theta=1$ ~and  ~$d=10000$.}%
\label{fig:ti_critfunc}%
\end{figure}

\subsection{Segmentation under dependence}\label{sec:segmentation_dependence}
 
Figure \ref{fig:ti_critfunc} shows one realization of ~$\tilde{t}(i)$'s from \eqref{eq:lasso_cusum} for different weighting schemes\footnote{For the sake of comparison we shift and rescale the curves by the transformation: $f(i)\mapsto \tilde{f}(i)=[f(i)-\min_{1\leq j <n}f(j)]/[\max_{1\leq j <n}f(j)-\min_{1\leq j<n}f(j)]$.} and the corresponding critical curves ~$\tilde{C}_n(i;u,r)$ ~from \eqref{eq:def_critical_function}, which are in probability the limits of the ~$\tilde{t}(i)$'s as ~$d\rightarrow\infty$. The vertical lines indicate the locations of maxima for the respective weightings, i.e. the positions of estimates ~$\hat{u}_\star$ ~and ~$\hat{u}$. We see that ~$w^{\text{exact}}$ ~provides a correct estimation of ~$u=70$ ~whereas ~$w^{\text{standard}}$ ~does not.  For ~$w^{\text{simple}}$ ~we estimate a more centered change point and ~$w^{\text{standard}}$ ~estimates a completely wrong location at the right border.

 Figures \ref{fig:u55} and \ref{fig:u90} demonstrate the performance of the estimates with respect to different weighting schemes. They compare the accuracy ~$P(\hat{u}_\star=u)$, the means and the standard deviation of the estimate ~$\hat{u}_\star$. We ran simulations with ~$\theta=1$ ~and ~$\tilde{\sigma}^2=9$ ~and considered a range of parameters: ~$\phi=-2,-1,-0.5,0,1$ ~and ~$d=50,150,250,200,\ldots,1950,2000$.

The figures show that the change point estimate based on the (estimated) exact weighting scheme ~$w^{\text{exact}}$ ~outperforms ~$w^{\text{standard}}$. The former estimates all changes correctly for arbitrary ~$\phi$ ~if we consider a sufficiently large number of panels ~$d$. 
The exact scheme is less biased and also has overall less variation for negative ~$\phi$. Furthermore, we see for chosen parameters that the distortion due to estimation with ~$\hat{w}^{\text{exact-banded}}_{\text{centered}}$ ~is rather mild\footnote{The results for ~$\hat{w}^{\text{exact-banded}}$ ~are close to ~$\hat{w}^{\text{exact-banded}}_{\text{centered}}$ ~in this particular setting and are therefore omitted.}. However, it might be strong for other parameters e.g. if the training period is too small. Notice that change point estimation with ~$\hat{w}^{\text{exact}}$ ~may (surprisingly) perform even better than with ~$w^{\text{exact}}$ (cf. Figure \ref{fig:u90}).

\begin{figure}%
\captionsetup[subfigure]{labelformat=empty}
\centering 
\vspace{-25pt}
\subfloat[][]{%
\label{fig:accuracy-fixed-55-standard}%
\includegraphics[width=0.4\textwidth, trim = 5mm 0mm 5mm 5mm]{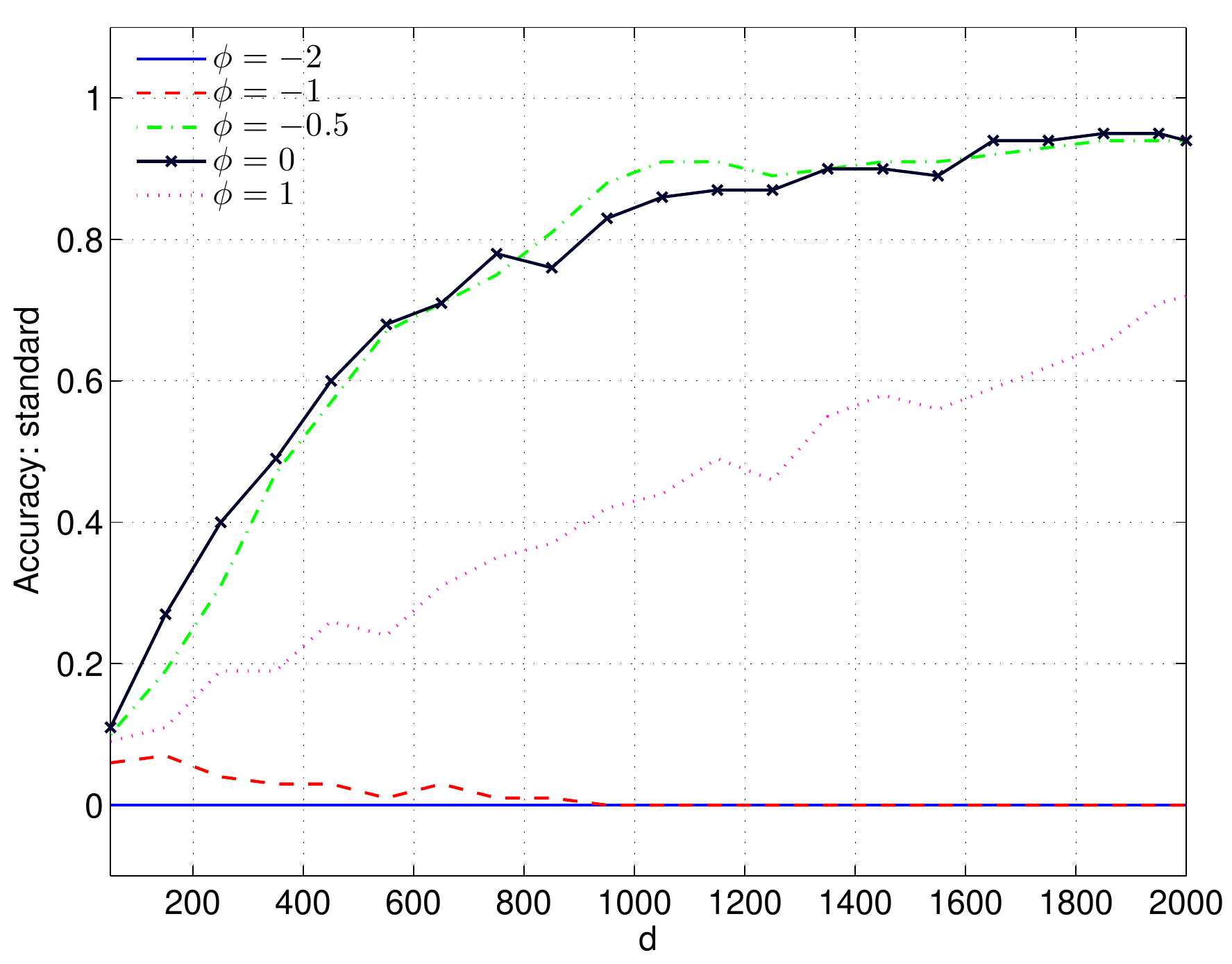}
}%
\hspace{40pt}%
\subfloat[][]{%
\label{fig:accuracy-fixed-55-exact}%
\includegraphics[width=0.4\textwidth, trim = 5mm 0mm 5mm 5mm]{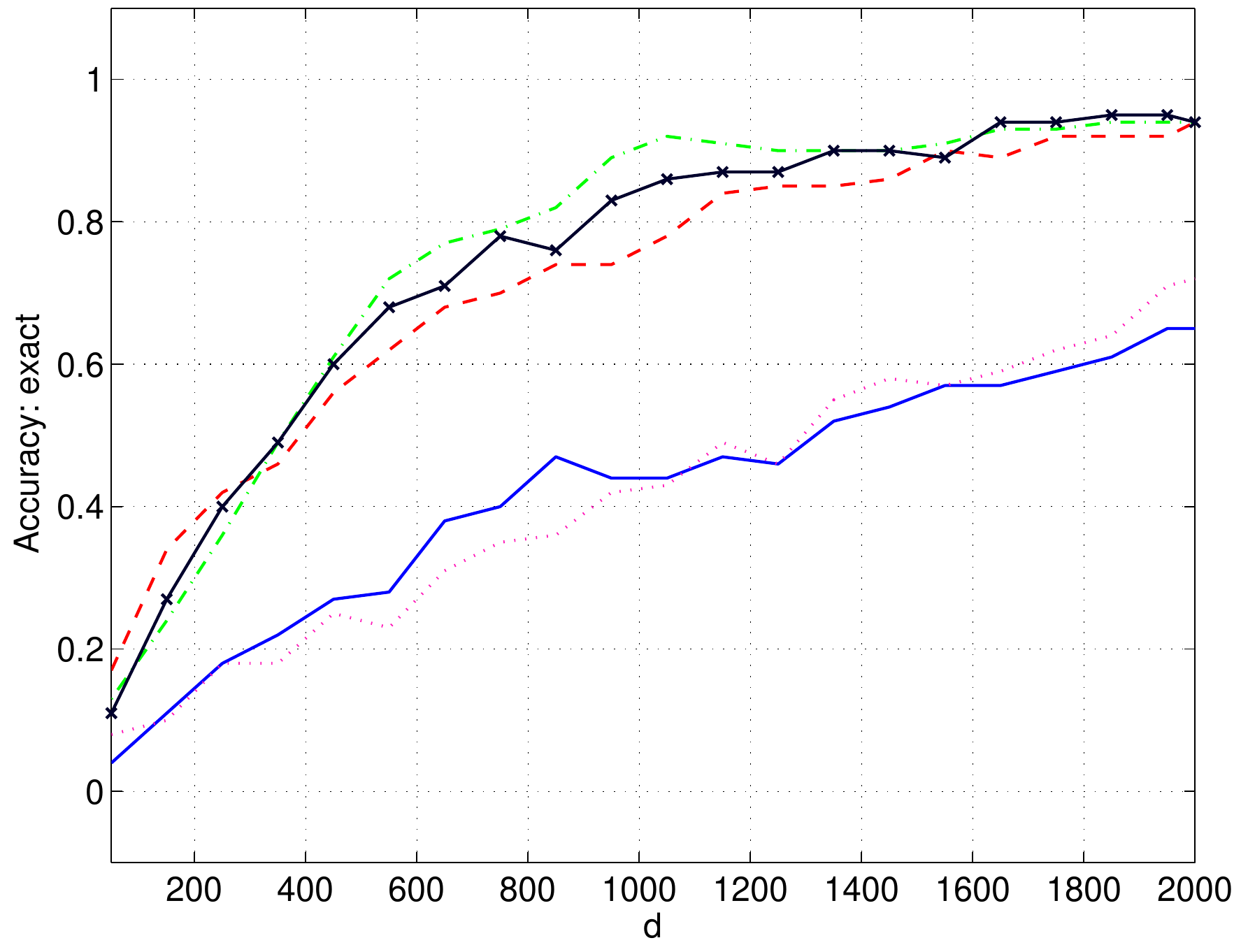}
}%
\\
%\vspace{-20pt}
\subfloat[][]{%
\label{fig:accuracy-fixed-55-exact-est}%
\includegraphics[width=0.4\textwidth, trim = 5mm 0mm 5mm 5mm]{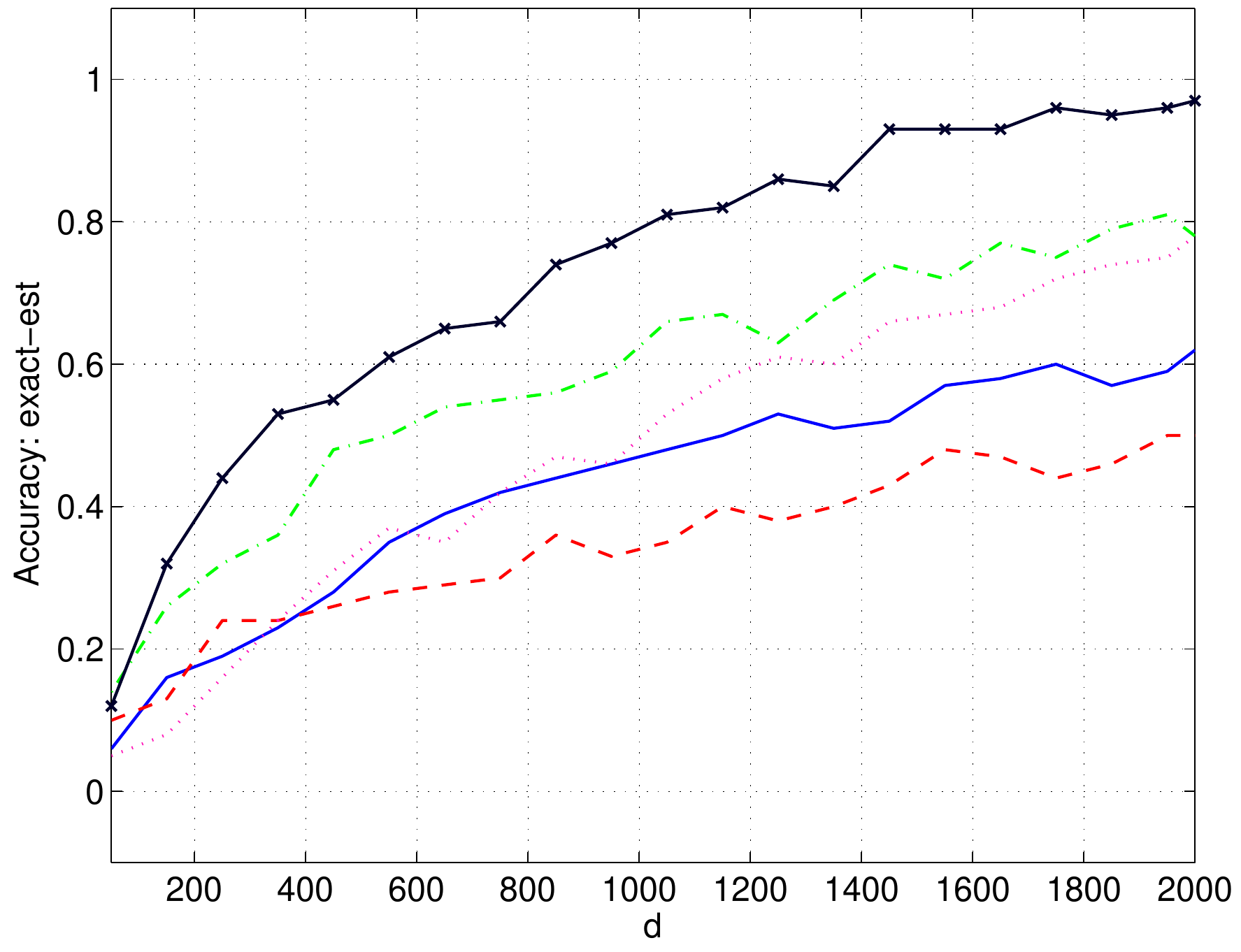}
}%
\hspace{40pt}%
\subfloat[][]{%
\label{fig:accuracy-fixed-55-exact-est-banded-center}%
\includegraphics[width=0.4\textwidth, trim = 5mm 0mm 5mm 5mm]{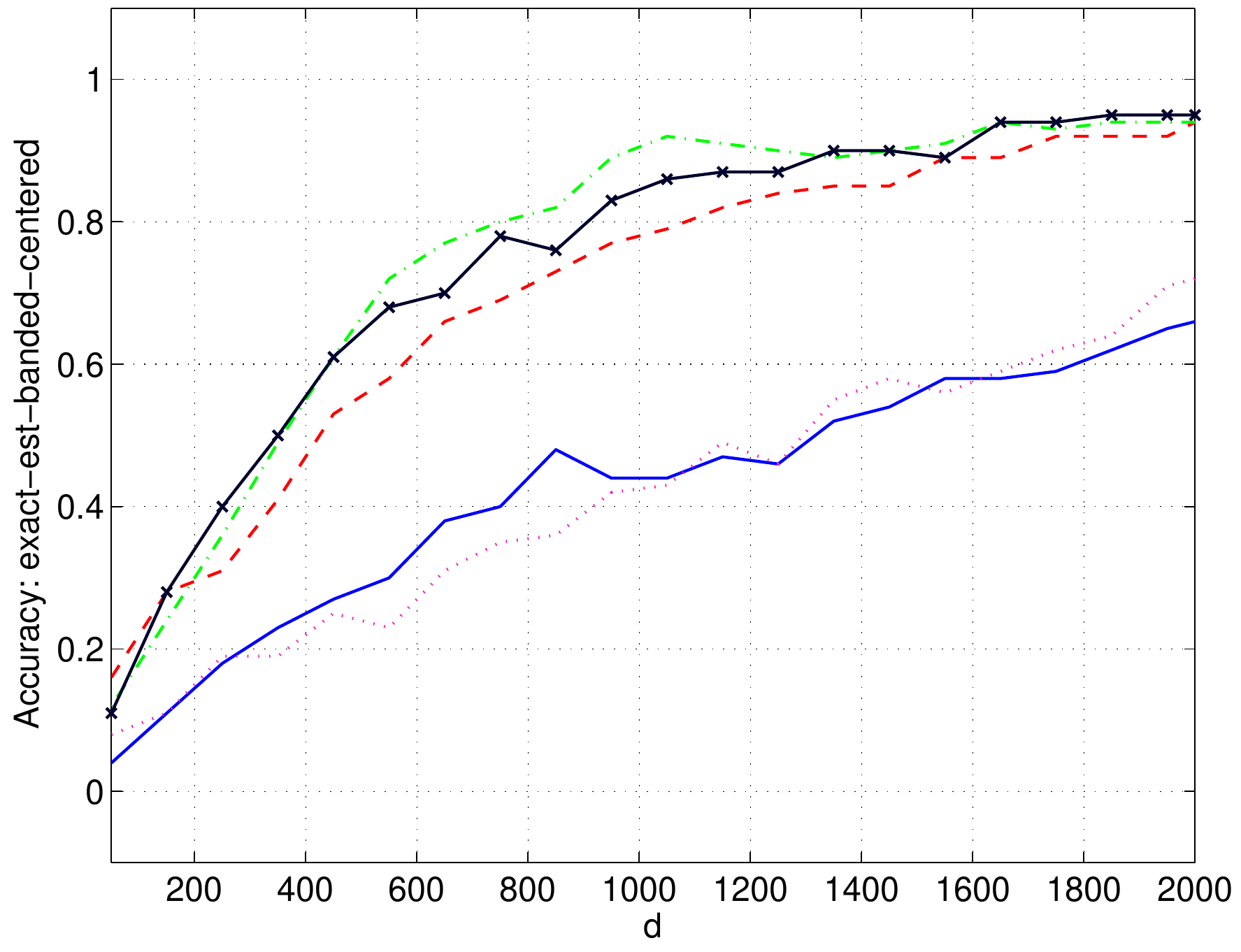}
}%
\\ 
\subfloat[][]{%
\label{fig:means-fixed-55-standard}%
\includegraphics[width=0.4\textwidth, trim = 5mm 0mm 5mm 5mm]{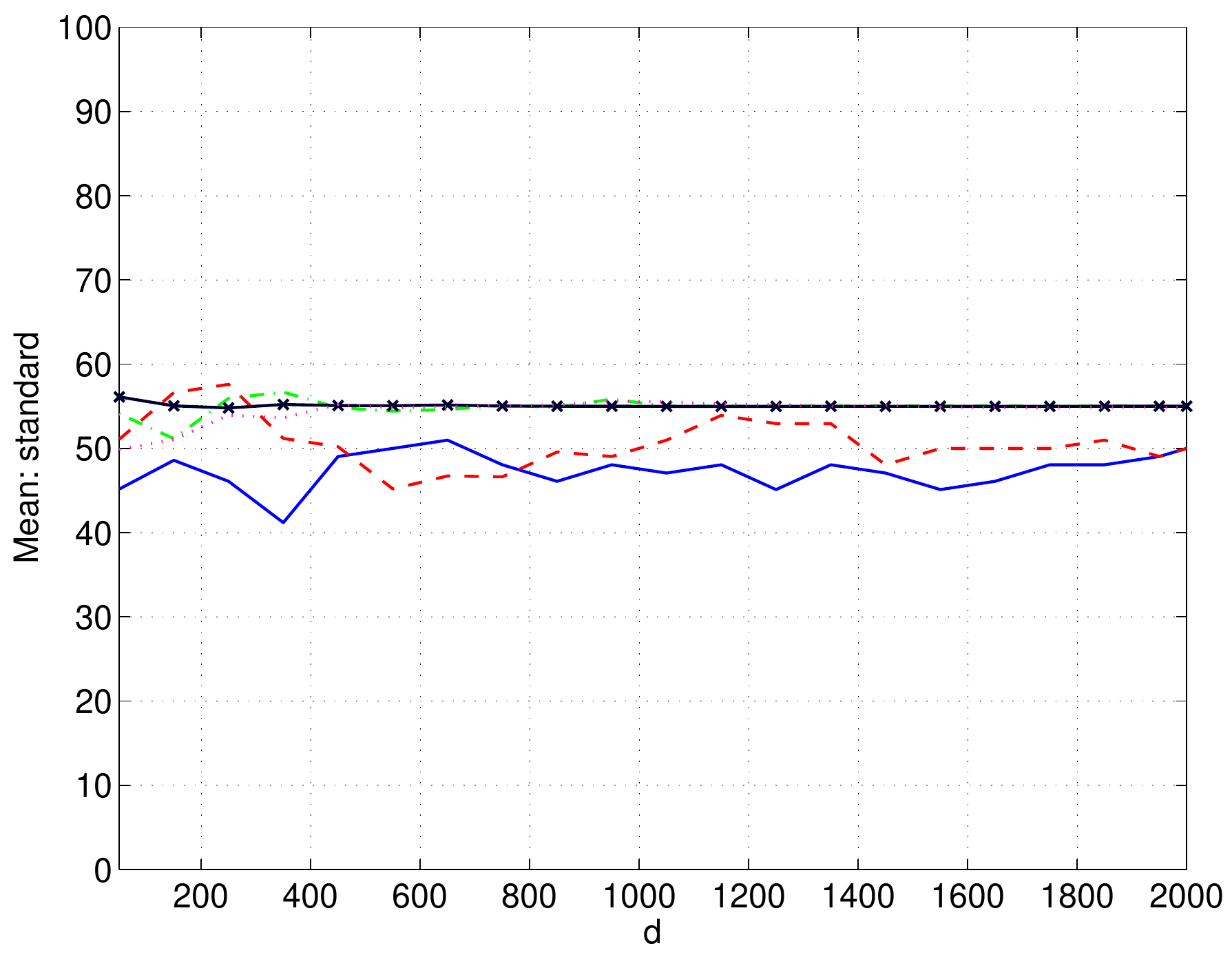}
}%
\hspace{40pt}%
\subfloat[][]{%
\label{fig:means-fixed-55-exact}%
\includegraphics[width=0.4\textwidth, trim = 5mm 0mm 5mm 5mm]{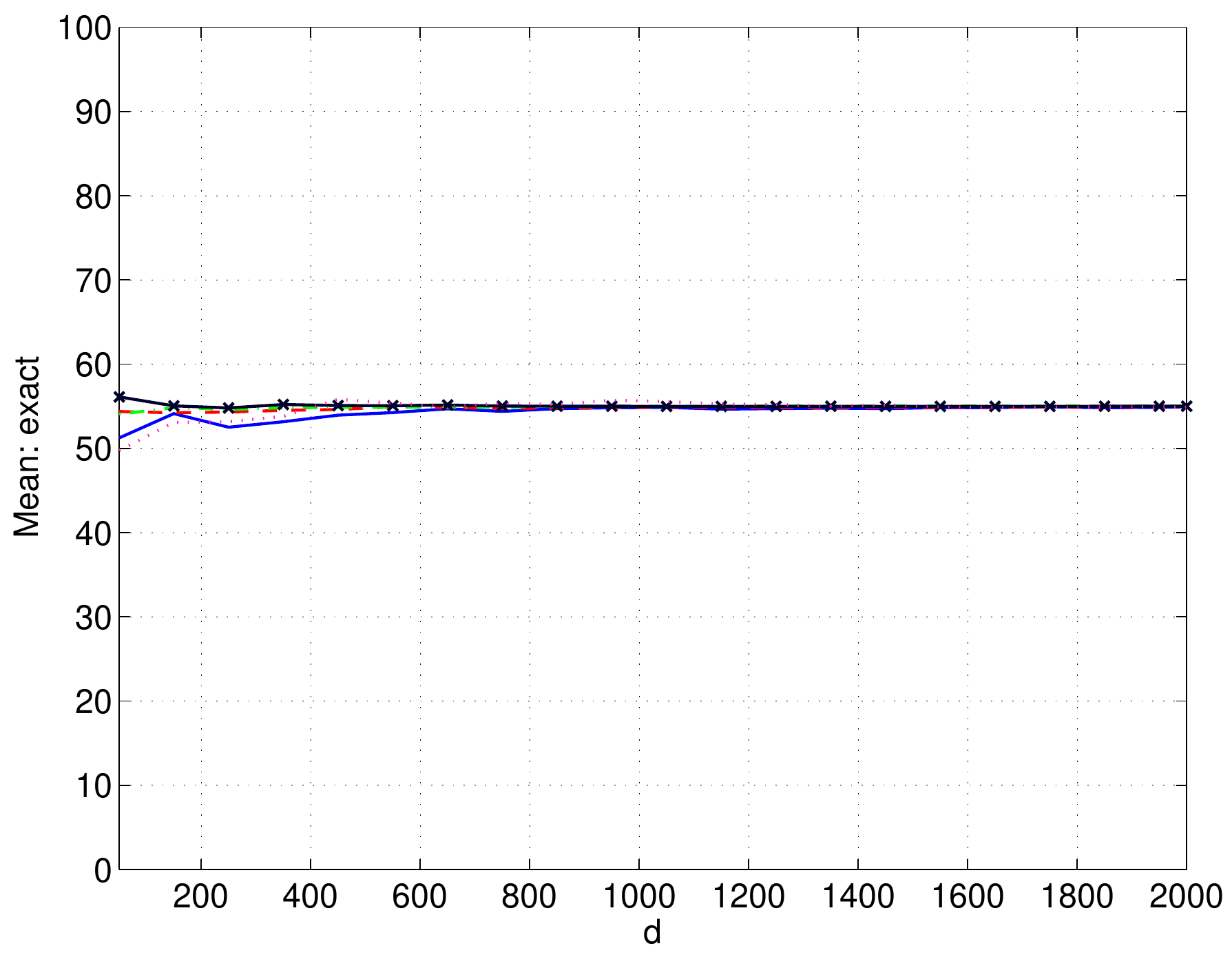}
}%
\\
%\vspace{-20pt}
\subfloat[][]{%
\label{fig:std-fixed-55-standard}%
\includegraphics[width=0.4\textwidth, trim = 5mm 0mm 5mm 5mm]{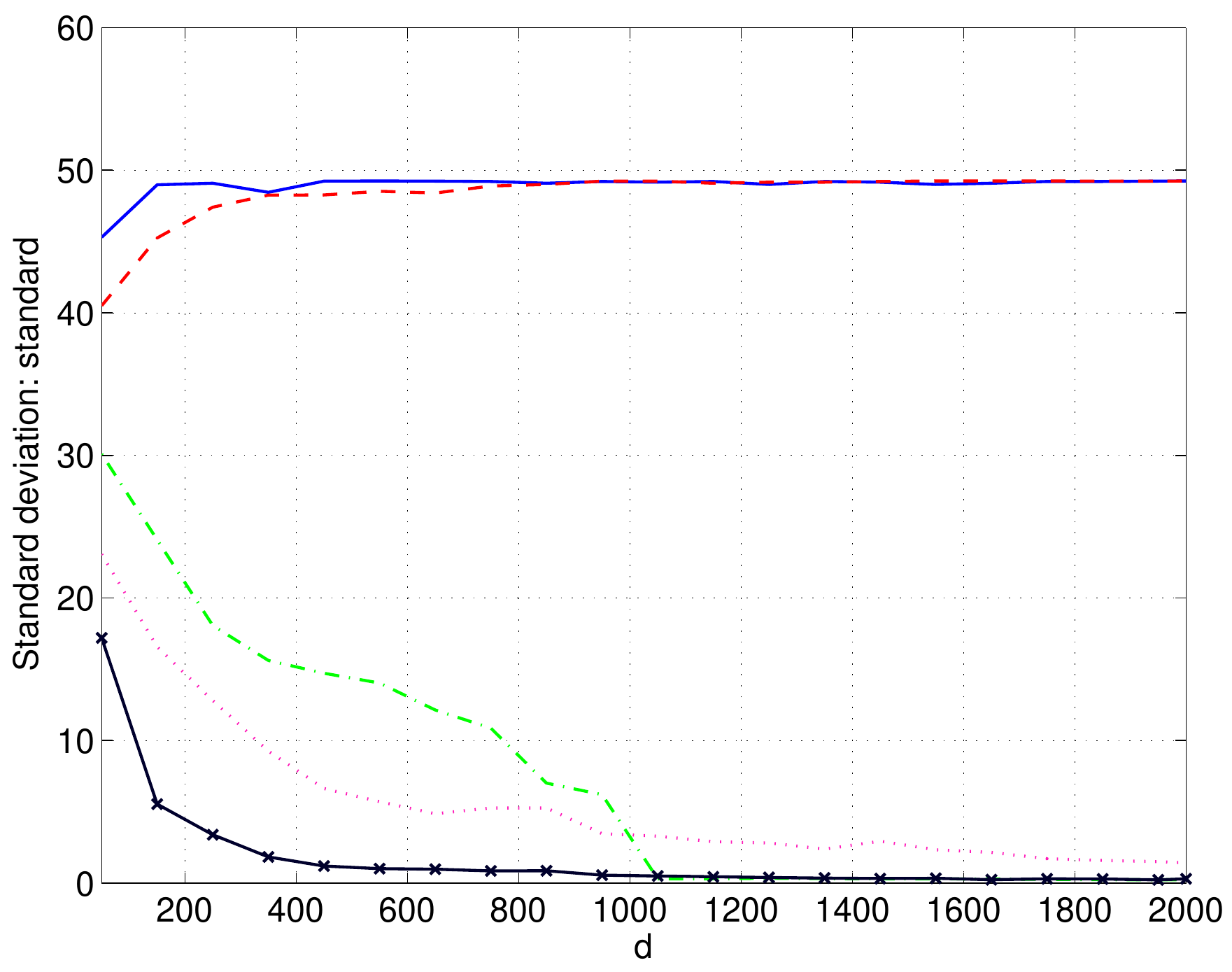}
}%
\hspace{40pt}%
\subfloat[][]{%
\label{fig:std-fixed-55-exact-est}%
\includegraphics[width=0.4\textwidth, trim = 5mm 0mm 5mm 5mm]{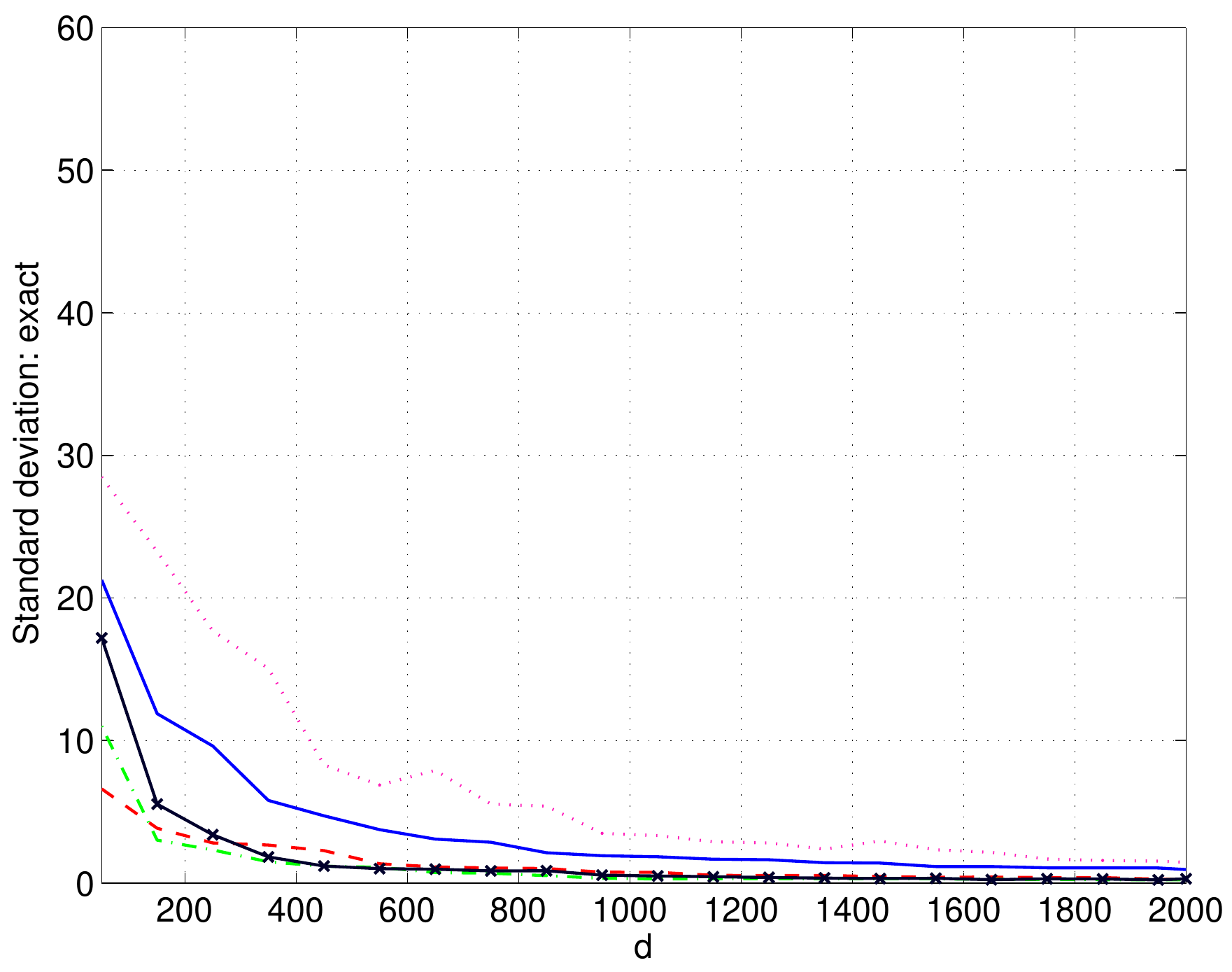}
}%
\caption{Accuracy, mean and standard deviation for ~$u=55$. ~$\hat{w}^{\text{exact}}$ ~is denoted by ``exact-est'' and ~$\hat{w}^{\text{exact-banded}}_{\text{centered}}$ ~is denoted by ``exact-est-banded-centered''.}%
\label{fig:u55}%
\end{figure}

\begin{figure}%
\captionsetup[subfigure]{labelformat=empty}
\centering 
\vspace{-25pt}
\subfloat[][]{%
\label{fig:accuracy-fixed-90-standard}%
\includegraphics[width=0.4\textwidth, trim = 5mm 0mm 5mm 5mm]{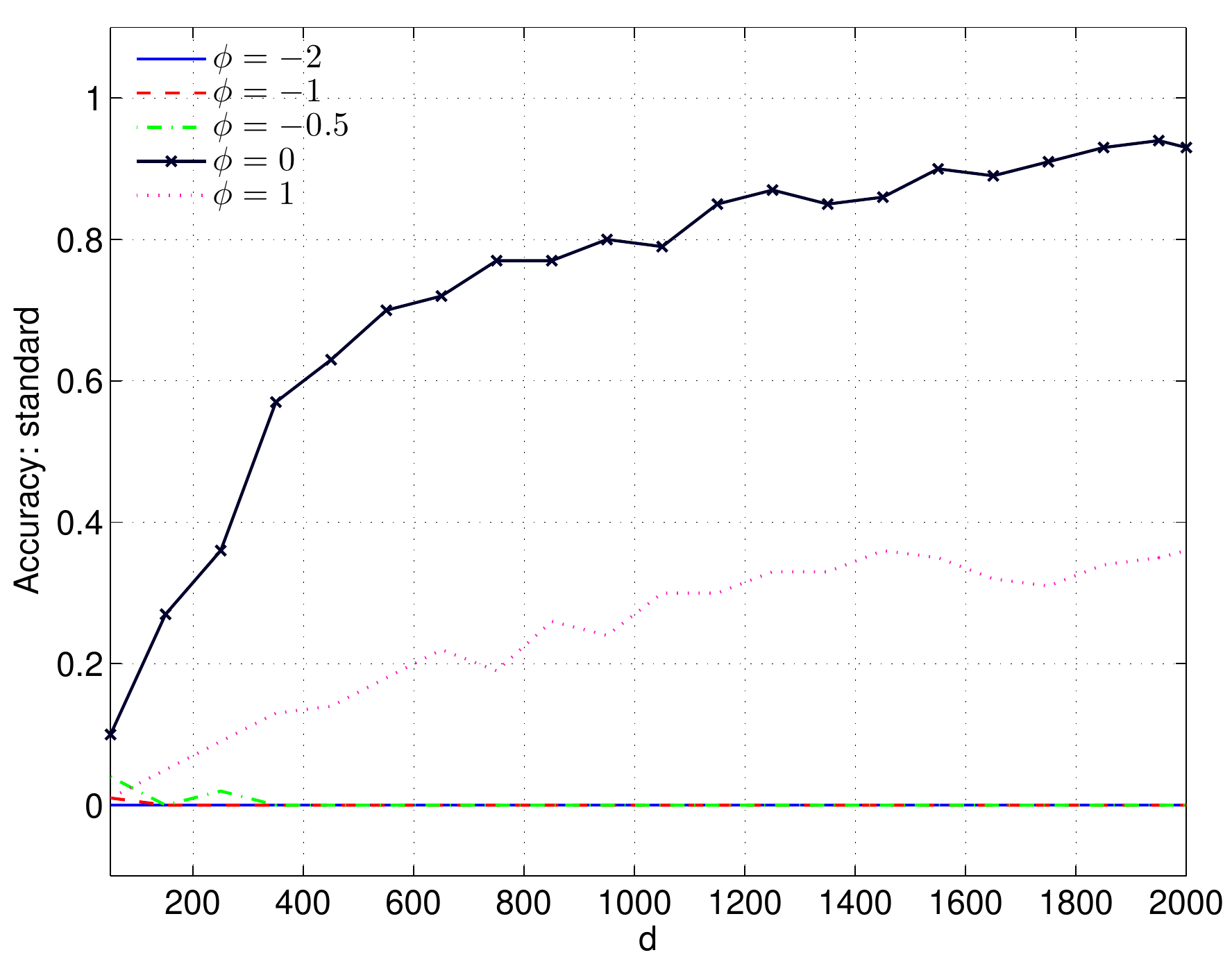}
}%
\hspace{40pt}%
\subfloat[][]{%
\label{fig:accuracy-fixed-90-exact}%
\includegraphics[width=0.4\textwidth, trim = 5mm 0mm 5mm 5mm]{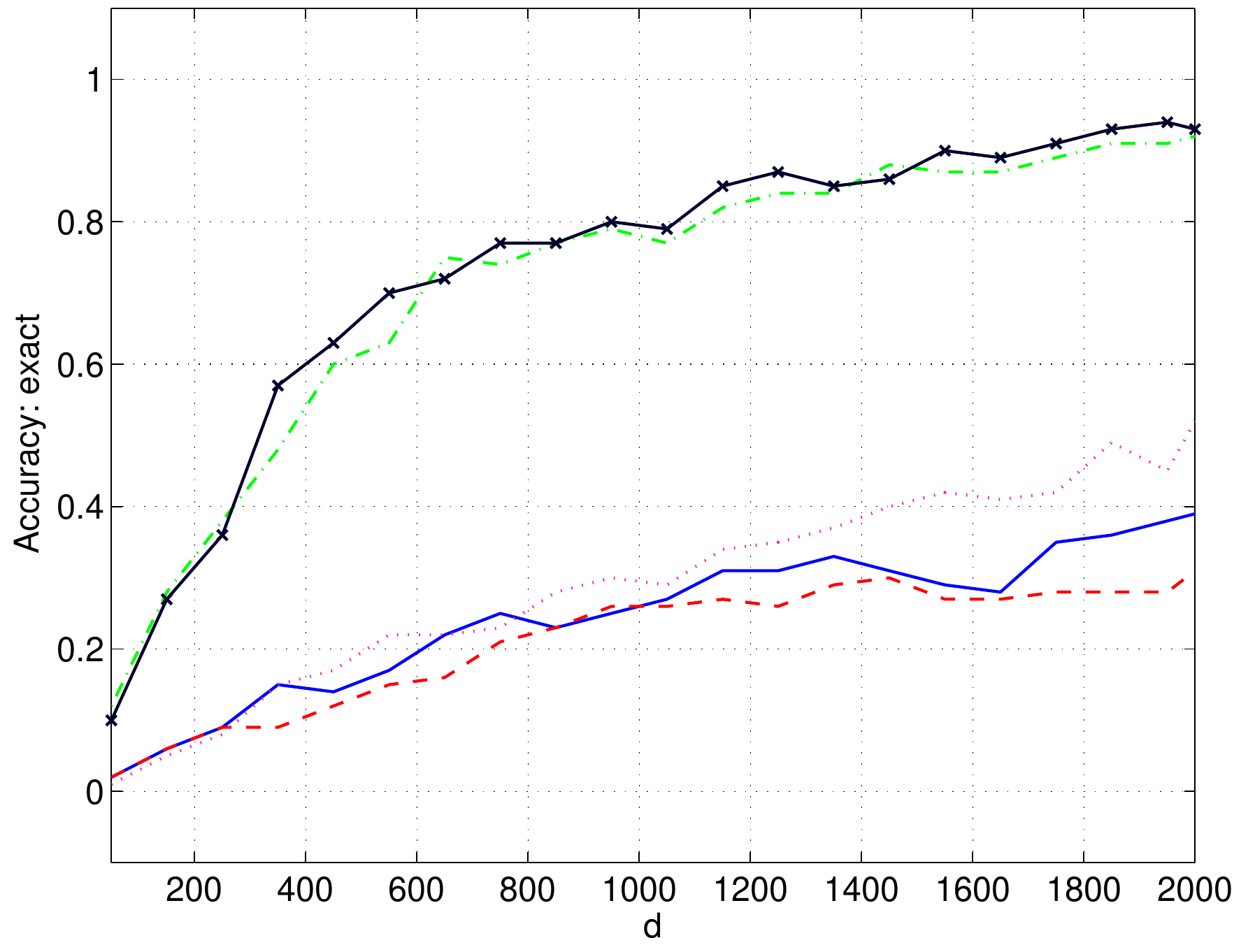}
}%
\\
%\vspace{-20pt}
\subfloat[][]{%
\label{fig:accuracy-fixed-90-exact-est}%
\includegraphics[width=0.4\textwidth, trim = 5mm 0mm 5mm 5mm]{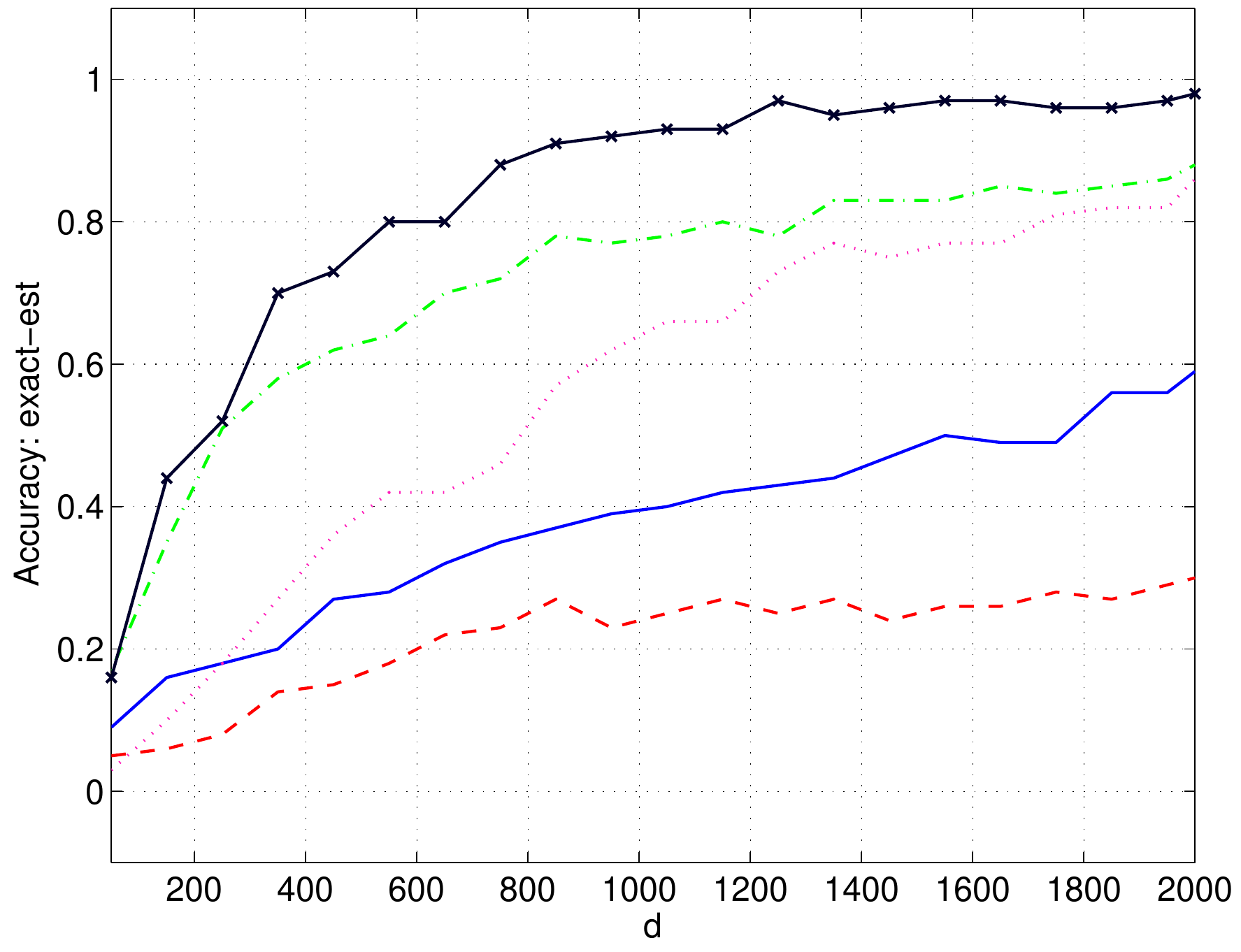}
}%
\hspace{40pt}%
\subfloat[][]{%
\label{fig:accuracy-fixed-90-exact-est-banded}%
\includegraphics[width=0.4\textwidth, trim = 5mm 0mm 5mm 5mm]{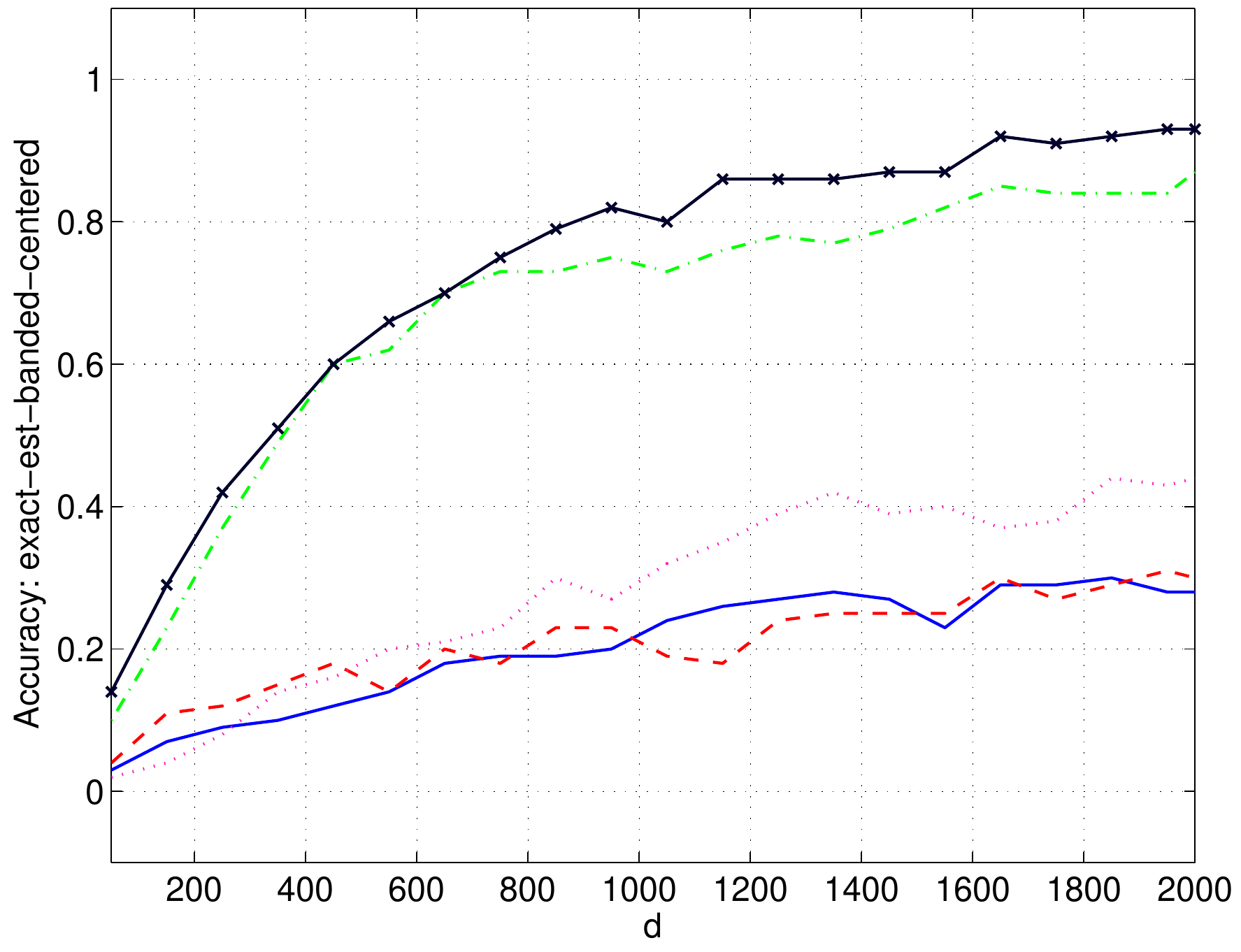}
}%
\\
%\vspace{-20pt}
\subfloat[][]{%
\label{fig:means-fixed-90-standard}%
\includegraphics[width=0.4\textwidth, trim = 5mm 0mm 5mm 5mm]{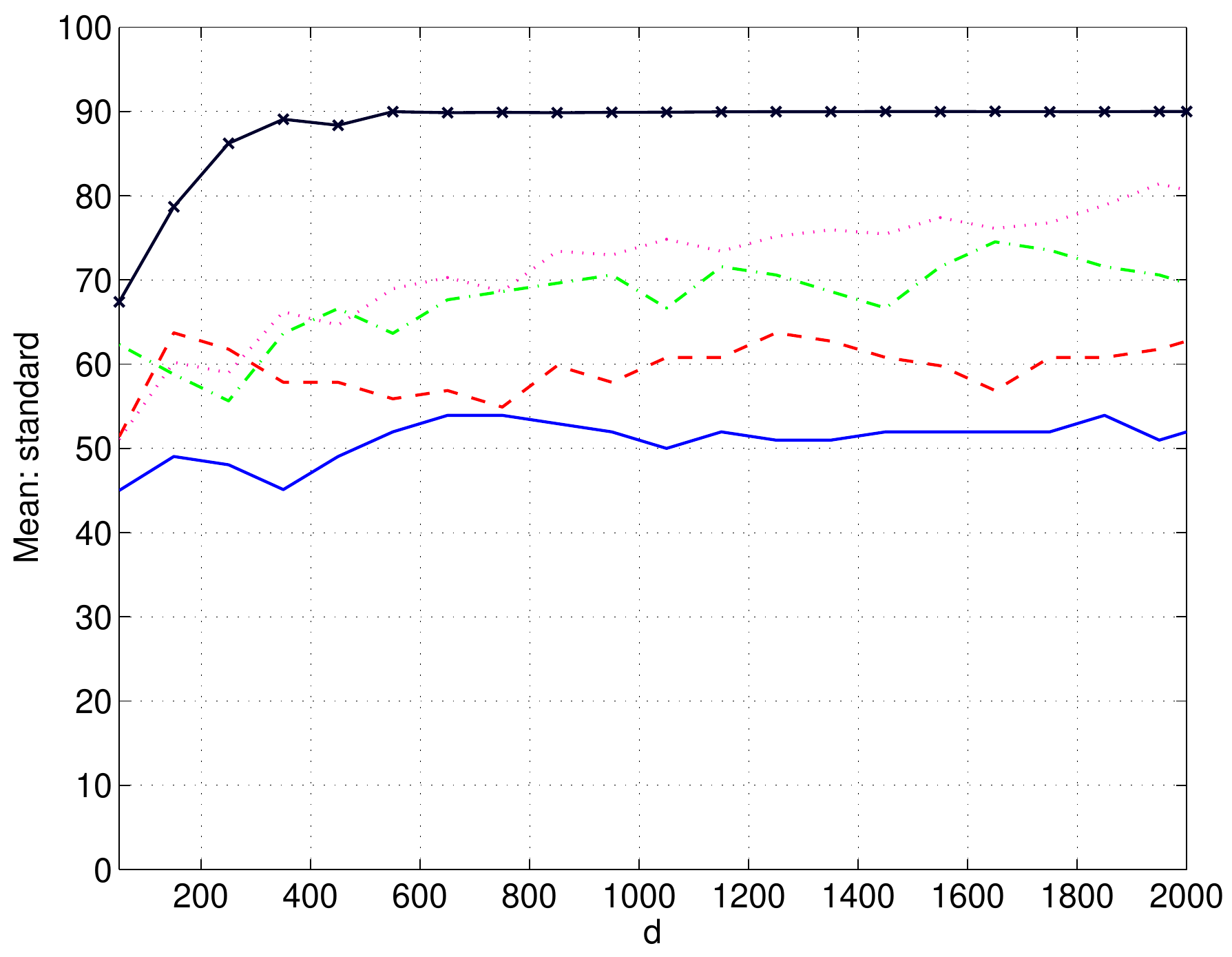}
}%
\hspace{40pt}%
\subfloat[][]{%
\label{fig:means-fixed-90-exact}%
\includegraphics[width=0.4\textwidth, trim = 5mm 0mm 5mm 5mm]{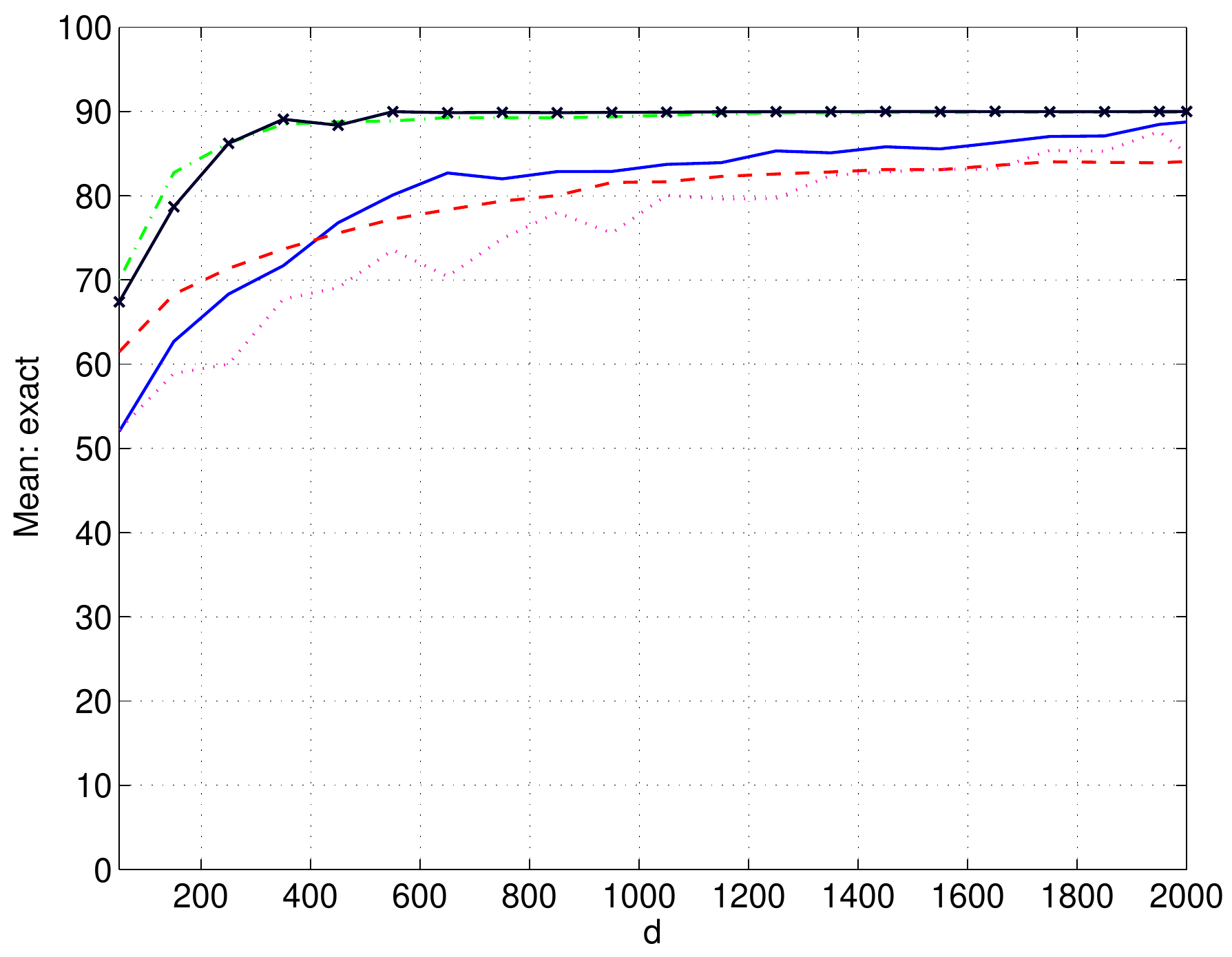}
}%
\\
%\vspace{-20pt}
\subfloat[][]{%
\label{fig:std-fixed-90-standard}%
\includegraphics[width=0.4\textwidth, trim = 5mm 0mm 5mm 5mm]{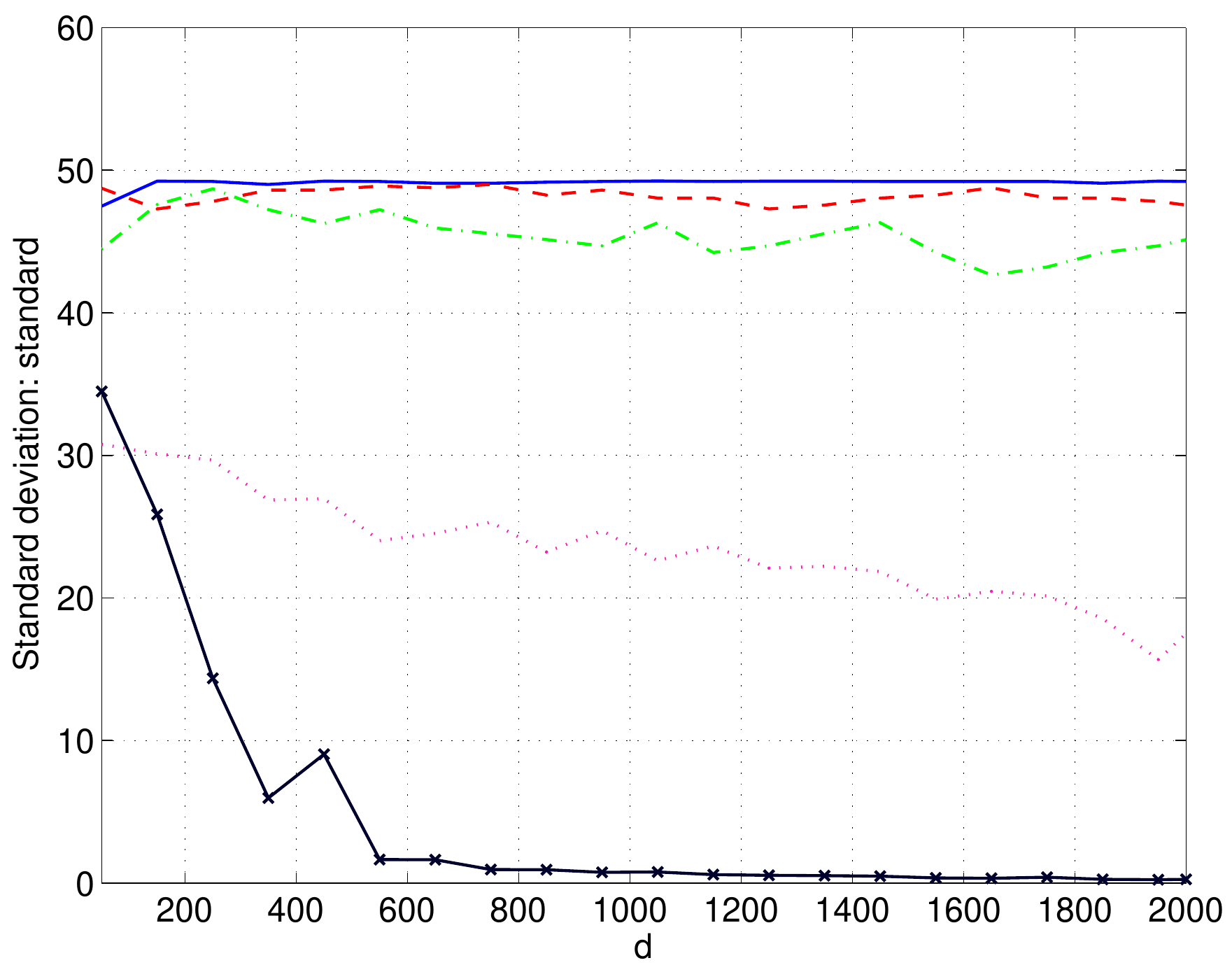}
}%
\hspace{40pt}%
\subfloat[][]{%
\label{fig:std-fixed-90-exact}%
\includegraphics[width=0.4\textwidth, trim = 5mm 0mm 5mm 5mm]{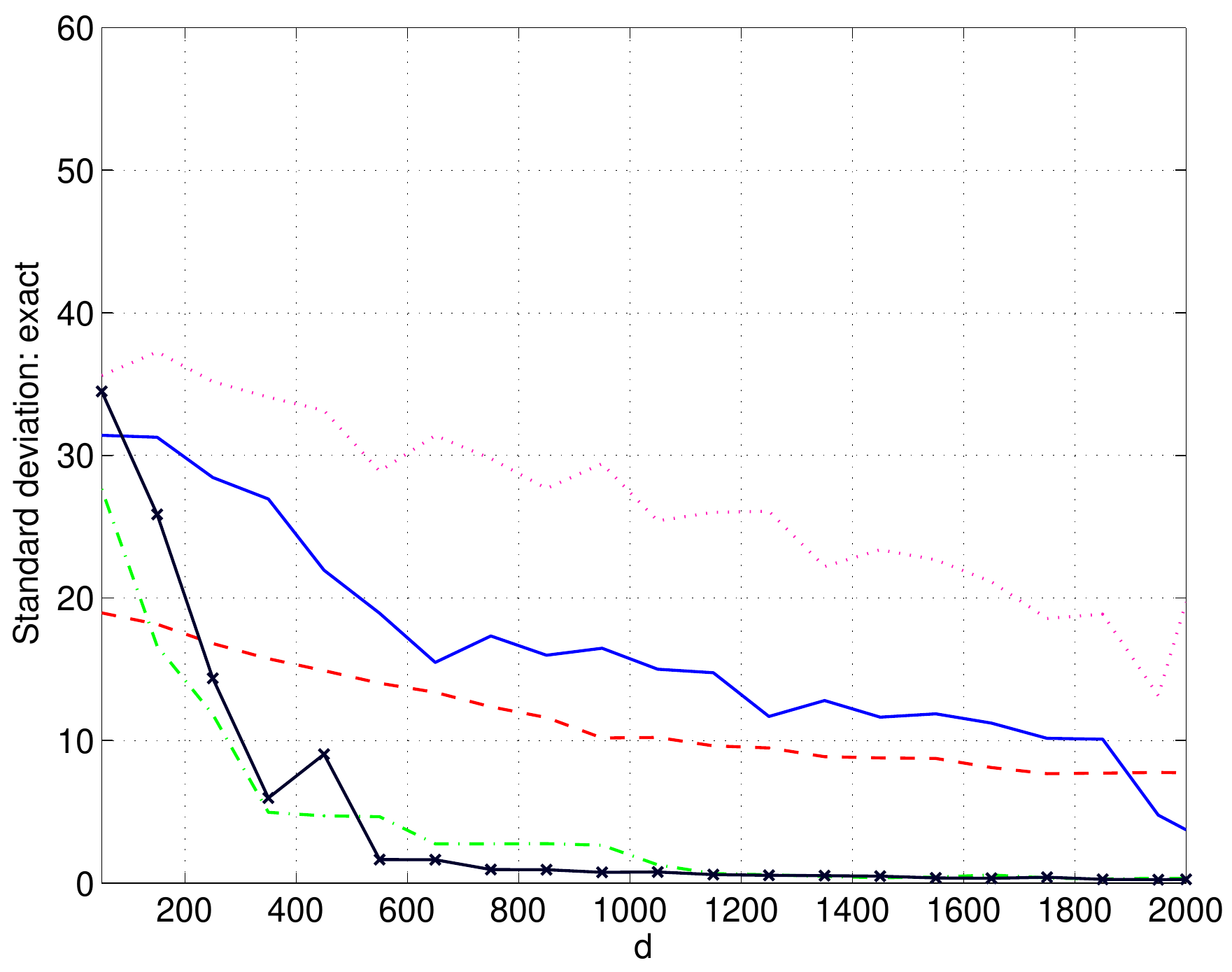}
}%
\caption{Accuracy, mean and standard deviation for ~$u=90$}%
\label{fig:u90}%
\end{figure}

\subsection{Segmentation under varying change point locations}

So far, we assumed a fixed change point at some time point ~$u$ ~across all panels. It seems more realistic to allow for some minor fluctuations (around time point ~$u$) of change points in different panels.  Therefore, \cite{vert2011a} studied the behaviour of their procedure under randomized change points theoretically and empirically. They considered changes across panels ~$k=1,\ldots,d$ ~that are located at random change points ~$u+U_k\in\{1,\ldots,n-1\}$, where ~$\{U_k\}$ ~are some i.i.d. random variables describing the fluctuations\footnote{$\{U_k\}$ are assumed to be also independent of ~$\{\varepsilon_{i,k},\zeta_i; \; i,k\in\mN\}$.}.  In \cite[cf. Theorem 4 and Figure 3]{vert2011a} they showed under appropriate assumptions that the standard weighting works also well in this setting in the sense that the probability ~$P(\hat{u}_\star\in u+S)$ ~tends to ~1 ~as ~$d\rightarrow\infty$, ~where ~$S$ ~is the support of ~$P_{U_1}$. 
We do not develop the theoretical analogue but show in Figure \ref{fig:vary} empirically that, as should be expected, the exact weighting tends to be beneficial under dependence. For this simulation we stick to the panels and simulation parameters of Subsection \ref{sec:segmentation_dependence} with ~$\tilde{\sigma}^2=9$ ~and ~$\theta=1$. As in \cite{vert2011a}, we assume ~$P(U_k=\pm 2)=0.5$ ~and we use the term {\it accuracy} now for  ~$P(\hat{u}_\star \in u+S)$. 
 
\begin{figure}%
\captionsetup[subfigure]{labelformat=empty}
\centering  
\subfloat[][]{%
\includegraphics[width=0.4\textwidth, trim = 5mm 0mm 5mm 5mm]{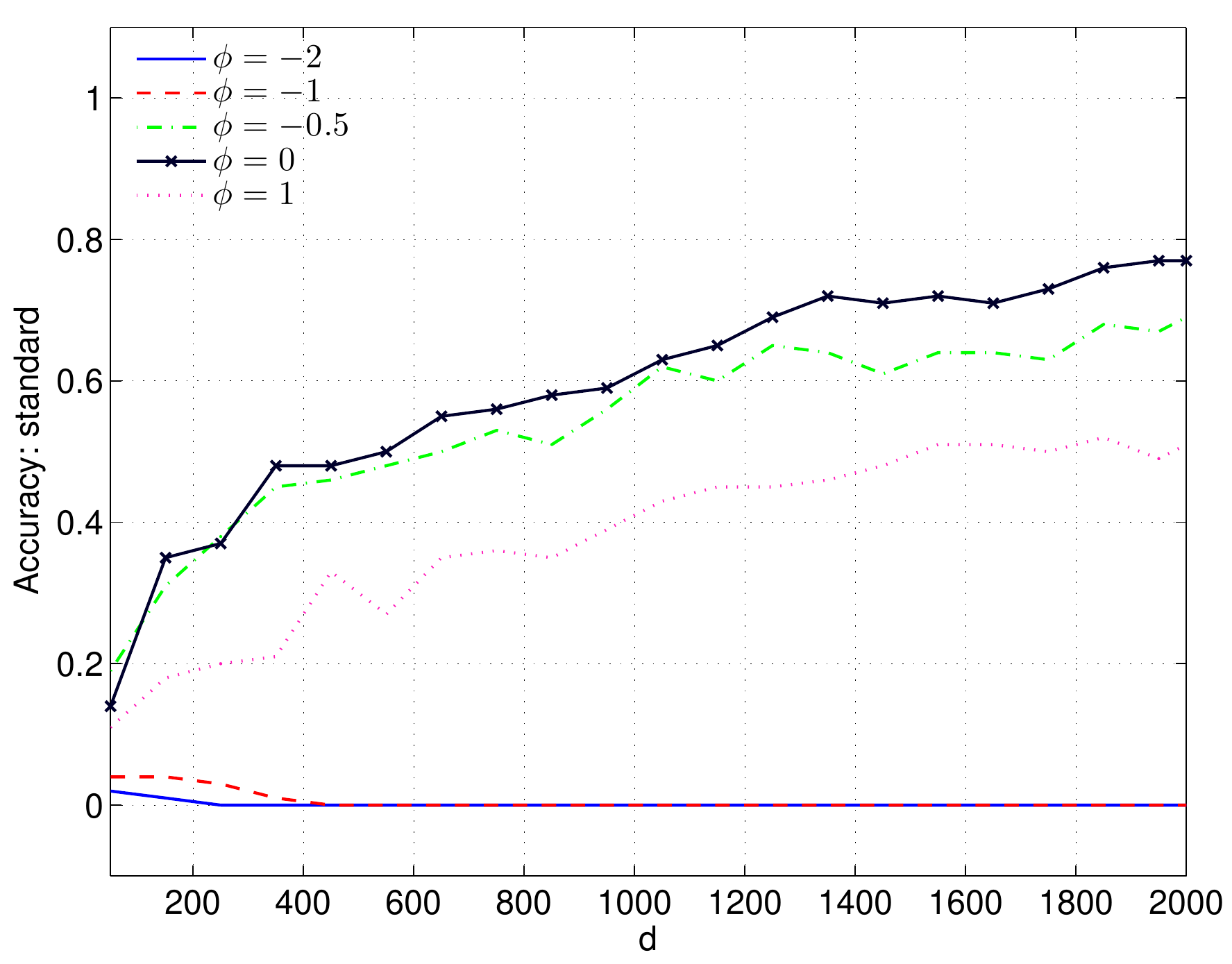}
}%
\hspace{40pt}%
\subfloat[][]{%
\includegraphics[width=0.4\textwidth, trim = 5mm 0mm 5mm 5mm]{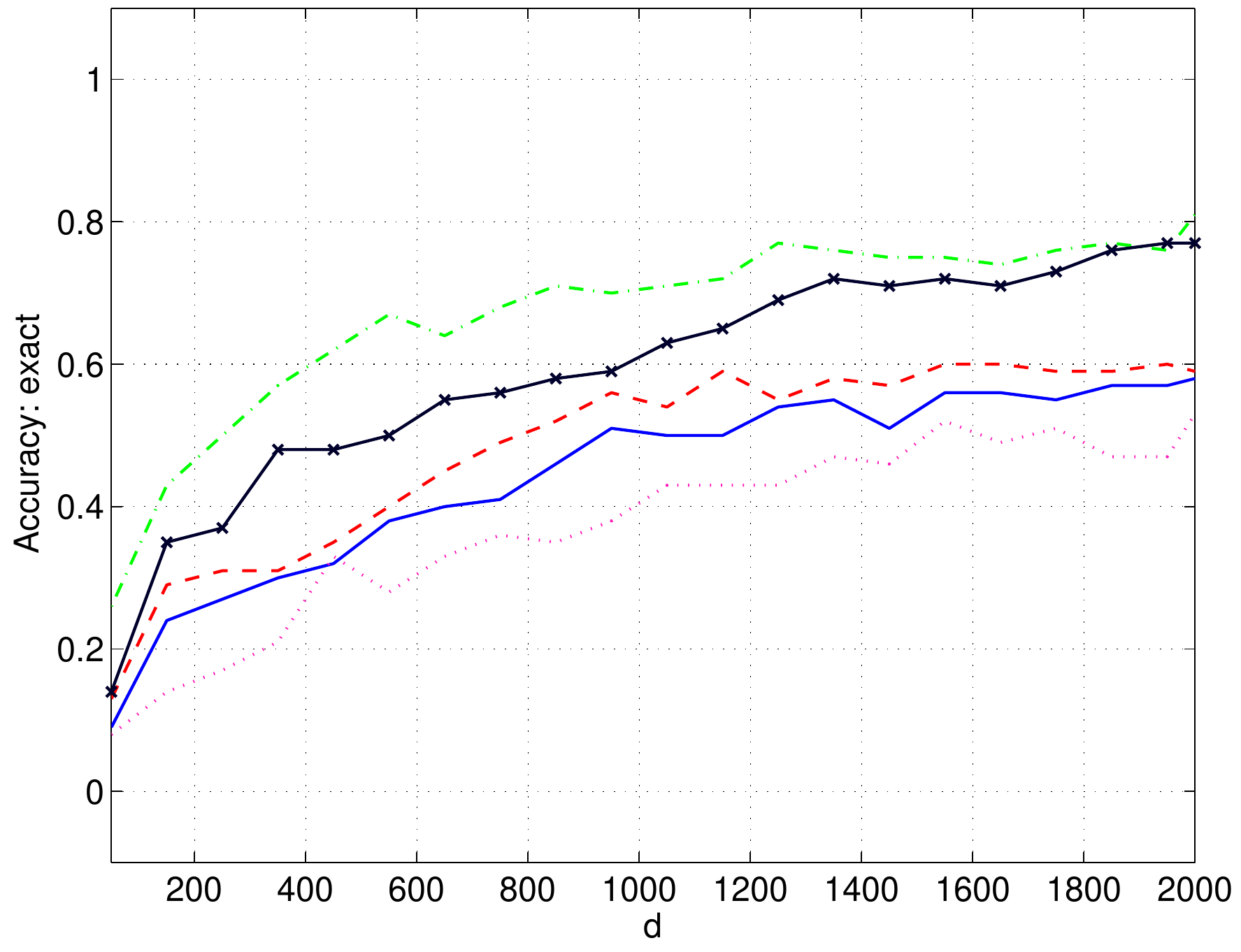}
}%
\caption{Random change locations ~$u+U_k$ ~that fluctuate around ~$u=70$. The means are ~$m_{j,k}=0$ ~for ~$j\leq u+U_k$ ~and ~$m_{j,k}=1$ ~for ~$j>u+U_k$ ~for all ~$k$.}
\label{fig:vary}%
\end{figure}

\subsection{Segmentation in the multiple change point scenario} 
In this subsection we assume multiple change points and compare the standard weighting scheme with the exact one using the denoising approach. First, we discuss {\it epidemic} changes, i.e. we have two change points ~$u_1$ ~and  ~$u_2$ ~where the means are temporarily shifted after ~$u_1$ ~but return to their former states after ~$u_2$.

We performed various simulations for different change point locations ~$u_1$ ~and ~$u_2$, moving average parameters ~$\phi$ ~and ~$\theta$ ~and for different variances ~$\tilde{\sigma}^2$ ~where we restricted our considerations to a simplified setting with the same magnitude of changes in all panels. We observed that the exact weighting tends to be beneficial in the sense that the overall picture improves. This is demonstrated in Figure \ref{fig:multiple_setting_epidemic}: With the exact weighting the accuracy for ~$\phi< -1$ ~increases considerably whereas for ~$\phi\geq -1$ ~it only decreases slightly. The curves are obtained using the fast LARS method but the group fused LASSO yields similar results.

In general multiple change point settings the situation is less clear than in the single change point or epidemic settings and the behaviour is rather erratic: The results strongly depend on the location and on the magnitude of the changes as well as on the moving average parameters - it is possible to find settings where the exact weighting scheme outperforms the standard one but also vice versa.

\subsection{Post processing estimated exact weights using regression} 

Here, our interest is again in the single change point scenario using ~$\hat{u}_{\star}$ ~with the simple weighting estimate ~$\hat{w}(i)=\hat{w}^{\text{exact}}(i,n)$ ~which is based on \eqref{eq:natural_estimate}.
We would like to mention a possible consistent modification which tends to be beneficial in situations described below and that may serve as a motivation for further research.  In the following we assume that ~$w^{\text{exact}}$ ~is strictly convex since the strictly concave case can be treated analogously.

Based on the results and the corresponding proofs of Section \ref{sec:segmentation} it seems reasonable to expect that, if ~$w^{\text{exact}}$ ~is strictly convex and smooth, which is the case for panels based on Example \ref{example:iid} or on Example \ref{example:MA1} (with e.g. $\phi\geq 0$), then modifications of ~$\hat{w}$ ~that are strictly convex, and therefore also less oscillating, should increase the precision of the resulting change point estimate. Obviously, the estimate ~$\hat{w}$ ~is usually not strictly convex due to the fluctuations around the strictly convex (discrete) function ~$w$. To obtain a smoother convex estimate ~$\hat{w}^{\text{exact-reg}}=\tilde{w}$, one may post-process the weights ~$\hat{w}$ ~using the well known least squares convex regression. The basic principle is that, given a regression model
\[
	\hat{w}(i)=w(i) + \varepsilon_i
\]
with a strictly convex function ~$w(i)$ ~and some centered noise sequence ~$\{\varepsilon_i\}$ ~for ~$i=1,\ldots,n-1$, we solve
\[
	\underset{\tilde{w}(i), g_i}{\text{Minimize}}\;  \sum_{i=1}^{n-1}\Big[\tilde{w}(i) - \hat{w}(i)\Big]^2,
\]
under the convexity restrictions
\begin{equation}\label{eq:convrestr}
	\tilde{w}(j)\geq \tilde{w}(i) + g_i (j-i)
\end{equation}
where  \eqref{eq:convrestr} holds true for ~$i,j=1,\ldots,n-1$ (cf., e.g., \citet{boyd2004} and \citet{hannah2013}). Clearly, in our situation, these estimates ~$\tilde{w}(i)$ ~remain consistent for ~$d\rightarrow\infty$ ~if the underlying original estimates ~$\hat{w}(i)$ ~were consistent and therefore strictly convex with probability tending to $1$, as ~$d\rightarrow\infty$.

In our simulations\footnote{The computation of the regression weights is carried out using the MATLAB software ``CVX: A system for disciplined convex programming'' (see \url{http://cvxr.com/cvx/}).} with moving average panels we observe that the weighting schemes ~$\hat{w}^{\text{exact-reg}}$ ~and ~$w^{\text{exact}}$ ~tend to yield similar estimation results  for ~$\hat{u}_\star$. They outperform ~$\hat{w}^{\text{exact}}$ ~for smaller variances ~$\tilde{\sigma}^2$, smaller panel numbers ~$d$ ~and for parameters ~$\phi$ ~that are closer to ~$-1$ ~(cf. Figure \ref{fig:regression_tildesigma}). This effect is stronger for change points ~$u$ ~closer to ~$n/2$ ~and weaker for ~$u$ ~closer to ~$n$. Notice that the estimate ~$\hat{w}^{\text{exact}}$ ~outperforms the ``true'' ~$w^{\text{exact}}$ ~for larger variances and larger panel numbers ~$d$. This goes in line with the observations of Subsection \ref{sec:segmentation_dependence}.

\begin{figure}[H]
\captionsetup[subfigure]{labelformat=empty}
\centering 
\subfloat[][]{%
\includegraphics[width=0.45\textwidth, trim = 5mm 0mm 5mm 5mm]{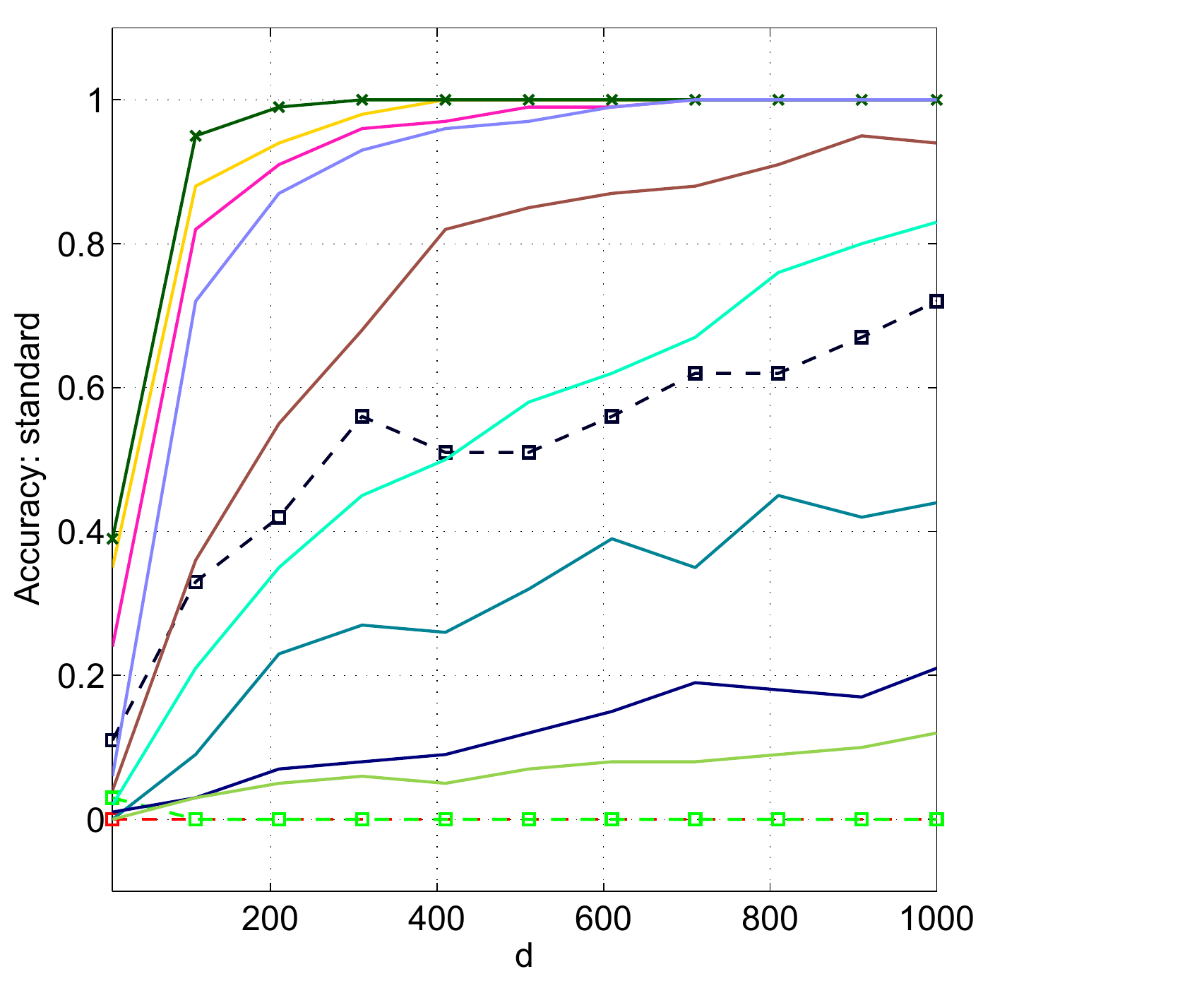}
}%
\subfloat[][]{%
\includegraphics[width=0.45\textwidth, trim = 5mm 0mm 5mm 5mm]{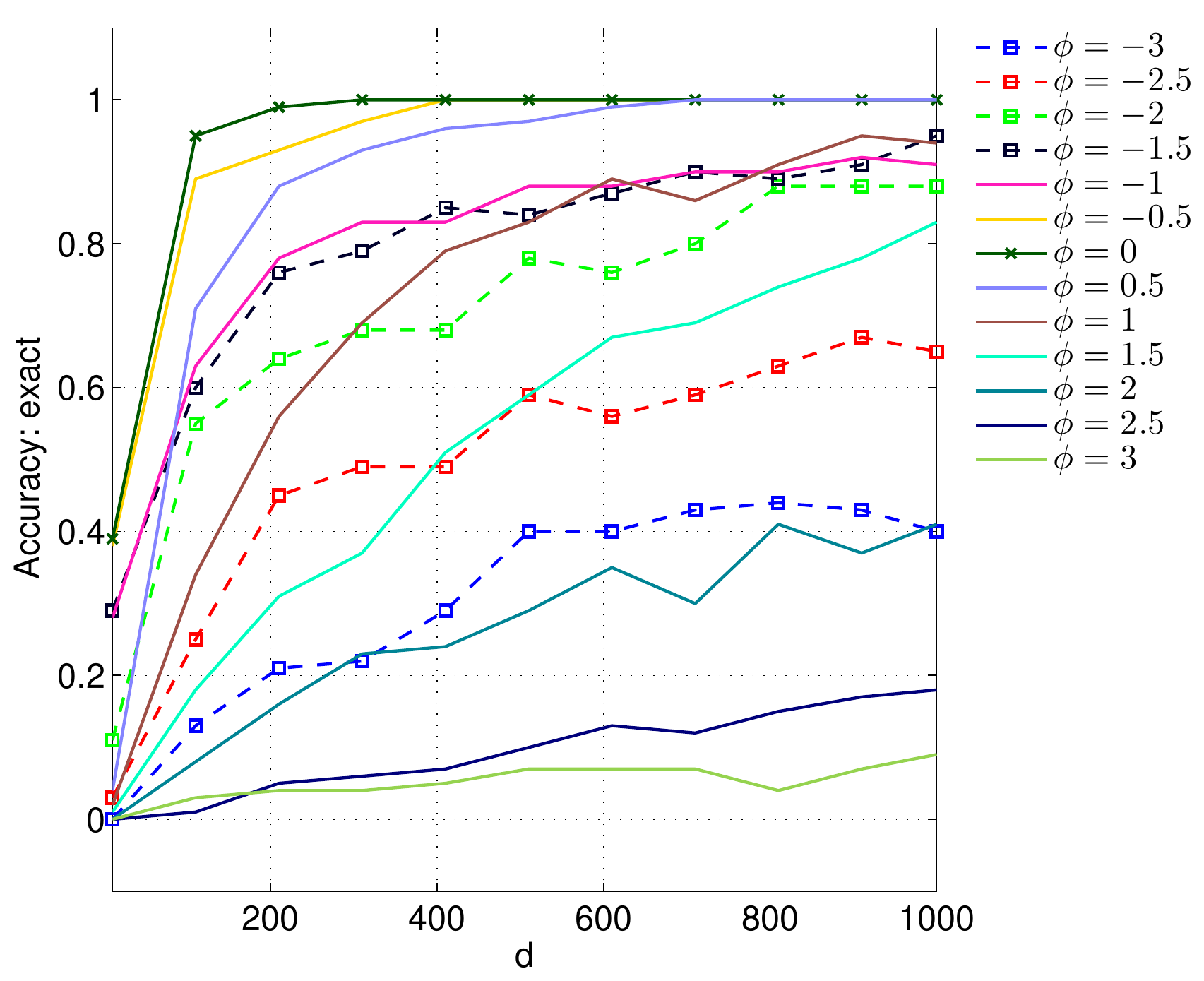}
}%
\caption{Epidemic change with change points at ~$\{55,80\}$ ~and a jump of size ~$+1$ ~in all panels. We consider ~$\tilde{\sigma}^2=1$ ~with ~$\theta=0$ ~and we do not take common factors into account, i.e. we set ~$\gamma_k=0$. {\it Accuracy} denotes now the probability that all change points are estimated correctly. Notice that the curves for ~$\phi=-2,-3.5$ ~are both zero in the left figure.}%
\label{fig:multiple_setting_epidemic}%
\end{figure}

\begin{figure}[H] 
%\captionsetup[subfigure]{labelformat=empty}
\centering
\subfloat[][$d=10$, $\phi=-0.3$]{% 
\includegraphics[width=0.4\textwidth, trim = 5mm 0mm 5mm 5mm]{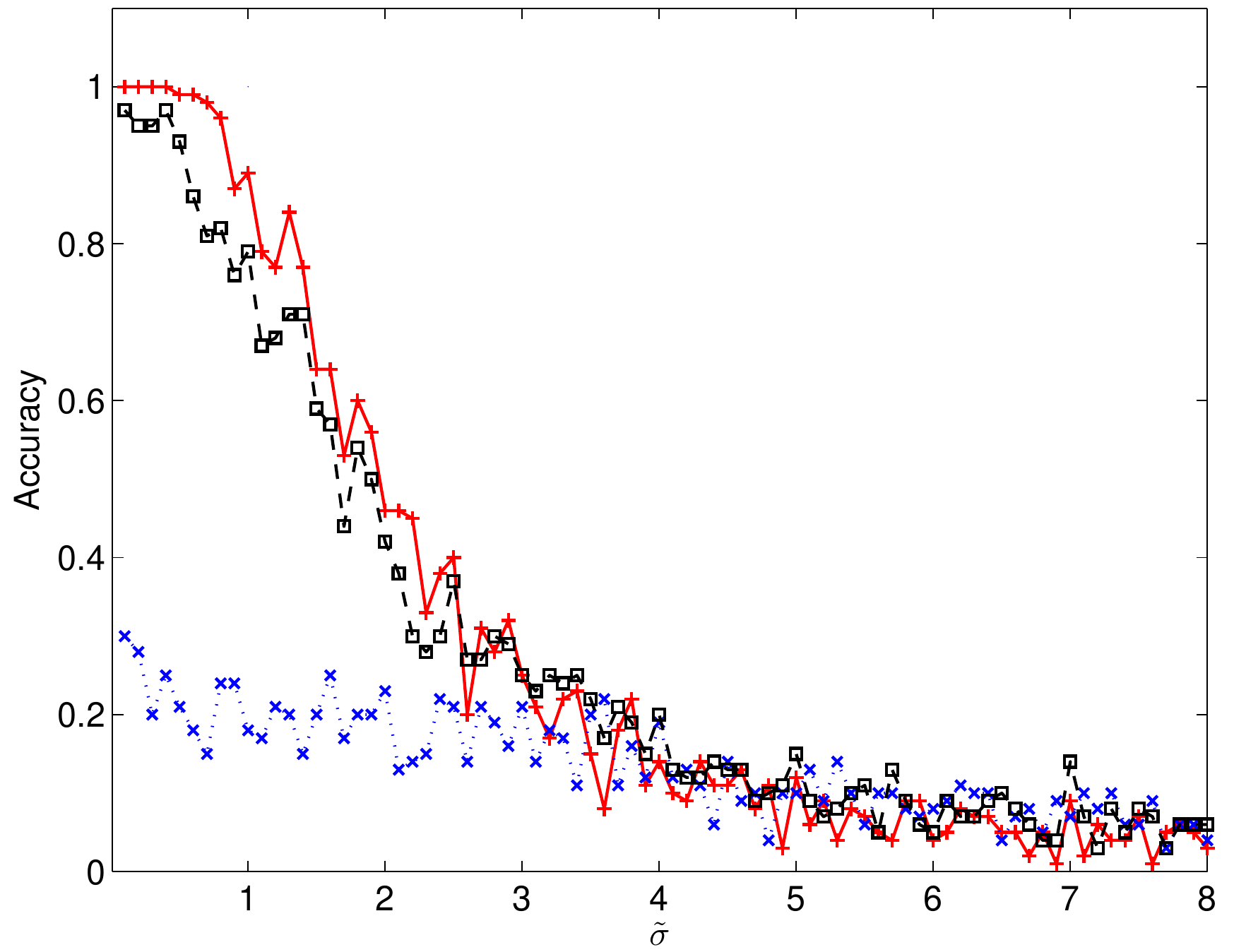}
}%
\hspace{40pt}%
\subfloat[][$d=200$, $\phi=-0.3$]{% 
\includegraphics[width=0.4\textwidth, trim = 5mm 0mm 5mm 5mm]{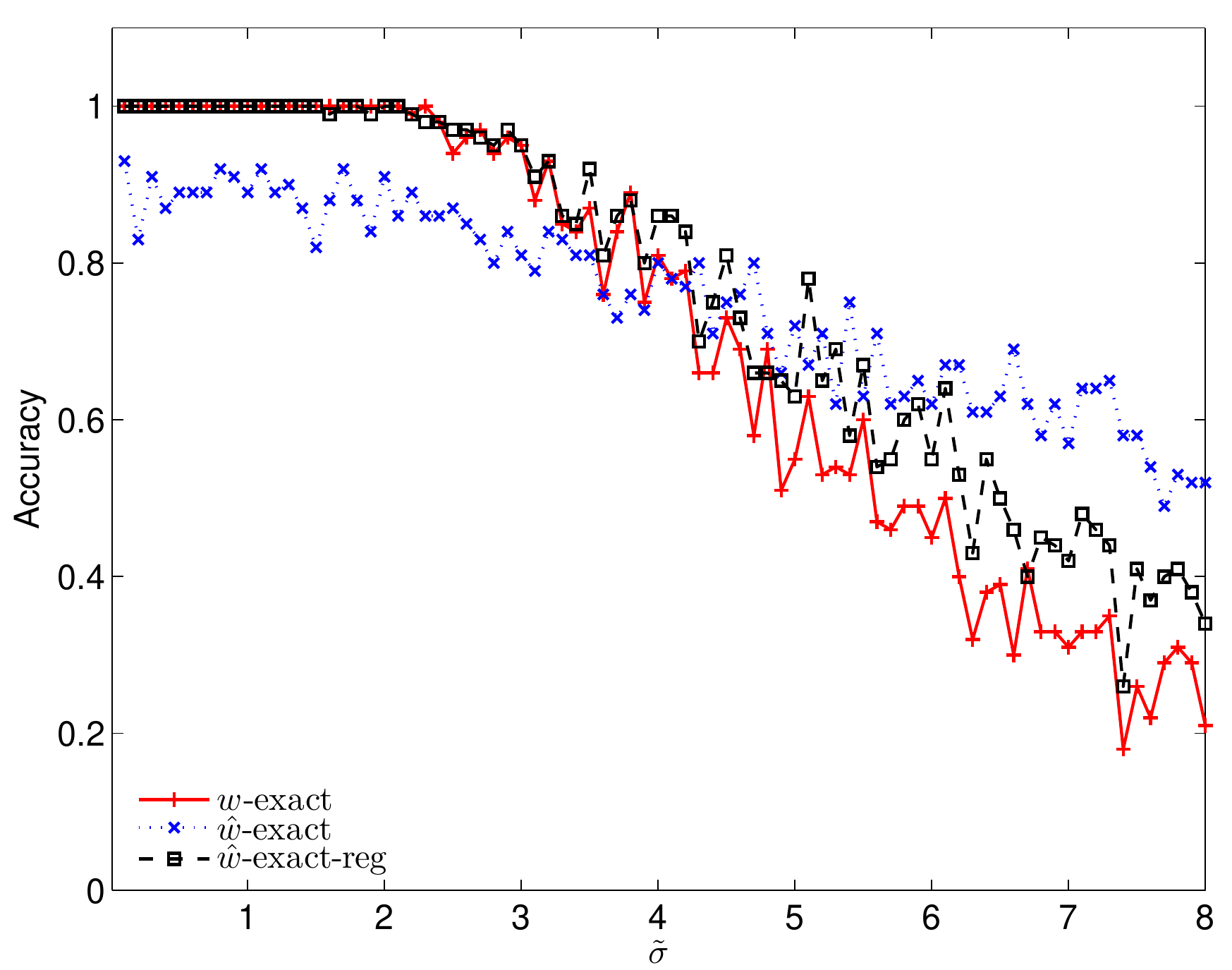}
}%
\\ 
\subfloat[][$d=10$, $\phi=0.5$]{% 
\includegraphics[width=0.4\textwidth, trim = 5mm 0mm 5mm 5mm]{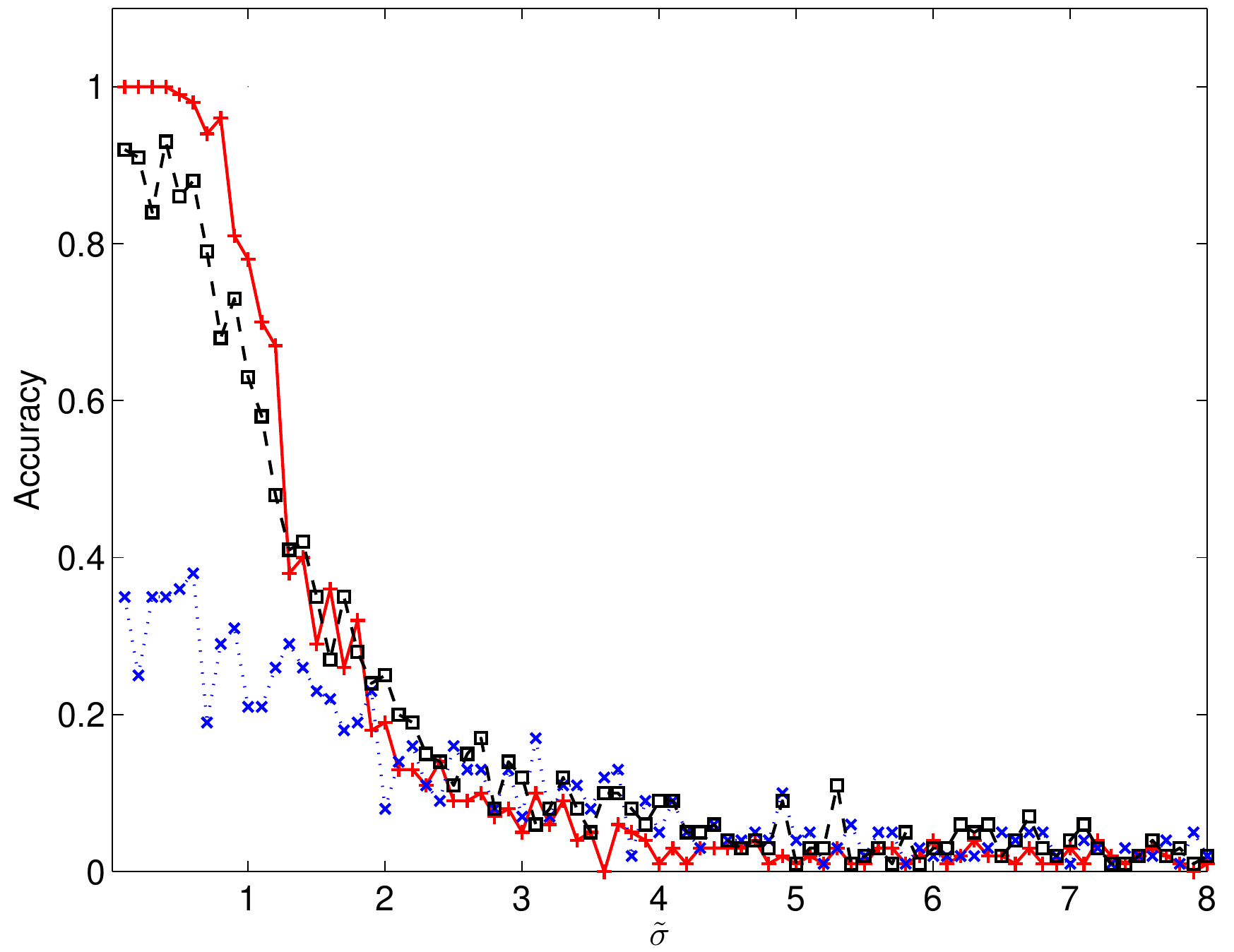}
}%
\hspace{40pt}%
\subfloat[][$d=200$, $\phi=0.5$]{% 
\includegraphics[width=0.4\textwidth, trim = 5mm 0mm 5mm 5mm]{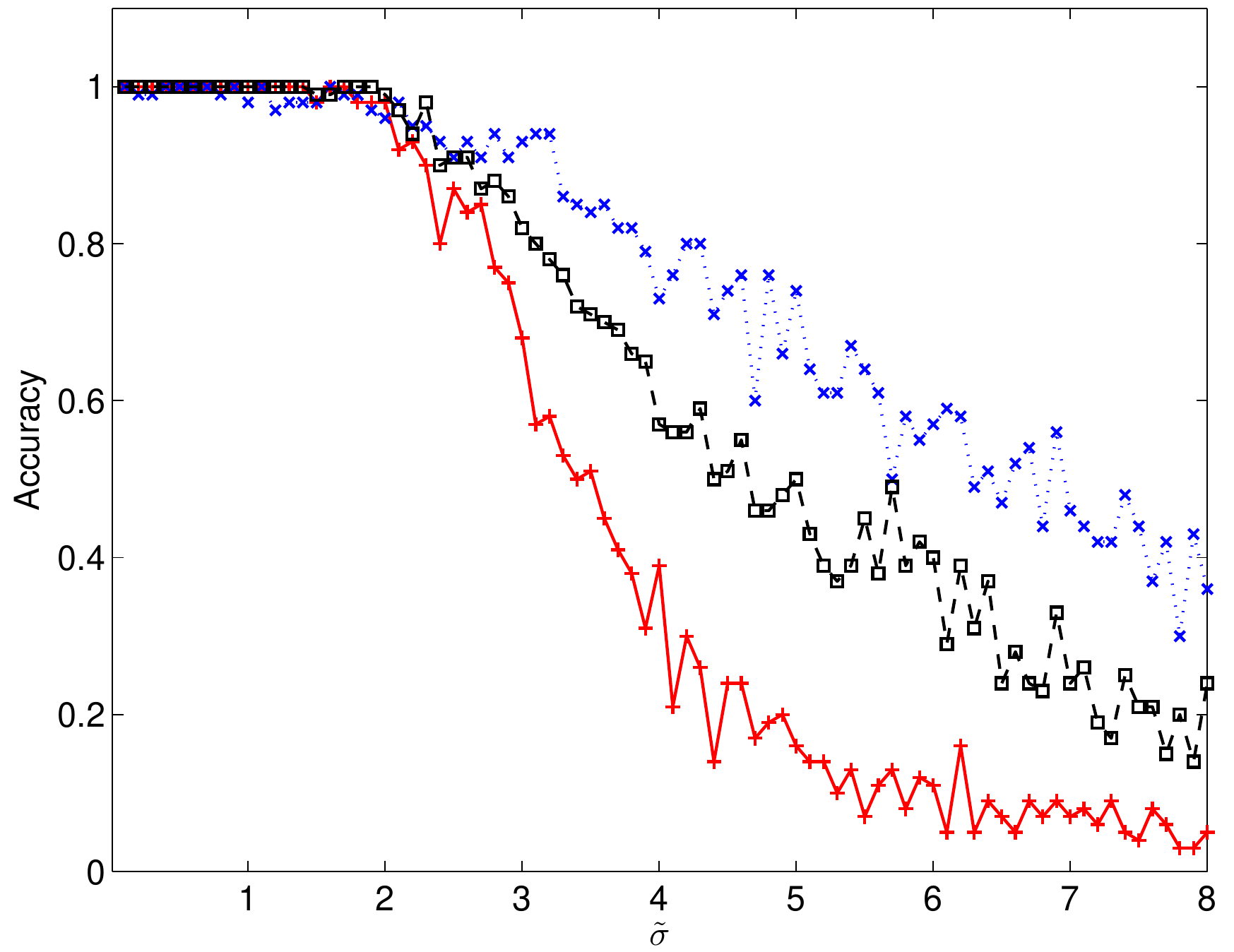}
}%
 \caption{We consider panels of length ~$n=50$ ~with a change point at ~$u=37$ ~and with ~$\theta=0$. We do not take common factors into account, i.e. we set ~$\gamma_k=0$.}

\label{fig:regression_tildesigma}%
\end{figure} 

%%% CONCLUSION %%%%%%%%%%%%%%%%%%%%%%%%%%%%%
\section{Conclusion}\label{sec:conclusion}
In this article we showed the connection of the total variation denoising approach of \cite{vert2011a} to the classical weighted CUSUM estimates. We generalized the consistency results of \cite{vert2011a} to panels of time series in the fixed ~$n$ ~and ~$d\rightarrow\infty$ ~setting under mild assumptions and studied consistency properties with respect to a well-known class of weighting schemes. Doing so, we also defined the criterion of {\it perfect estimation}, which is fulfilled in the independent setting if one uses the standard weighting, and showed that generally only a suitable covariance-dependent modification of these weights ensures this criterion. Thus, corresponding estimates outperform the standard weighting in various situations. We discussed appropriate estimation of these new weights and confirmed our results in a detailed simulation study. Moreover, we discussed the implications and possible advantages for the multiple change points and the random change point scenarios as well.

%%% PROOFS %%%%%%%%%%%%%%%%%%%%%%%%%%%%%%%
\newpage
\section{Proofs}\label{sec:proofs}
 
We start by proving Proposition \ref{prop:LASSO_CUSUM}, which requires some preliminary considerations and some key facts from  \citet{vert2011a}. For theoretical investigations and also for practical purposes the authors pick up the idea of \citet{leduc2008} and reformulate the minimization Problem \eqref{eq:totvar} as a {\it group fused} LASSO. 
\\
\\
Using the weights ~$w(i,n)$, they introduce a fixed design matrix ~$D\in\mR^{n\times (n-1)}$ ~with
\[
	D_{i,j}=
\begin{cases}
w(j,n),& i>j,\\
0, &\text{else}
\end{cases}
\] 
and set
\begin{align}
\begin{split}\label{eq:substitute} 
    \beta_{i,\bullet}&= \frac{U_{i+1,\bullet}-U_{i,\bullet}}{w(i,n)},
\end{split}
\end{align}
which can be compactly rewritten as ~$U= \mathds{1} U_{1,\bullet} + D\beta$ ~where ~$\mathds{1}=[1,\ldots,1]^T$. Thereby, the original problem \eqref{eq:totvar} transforms to
\begin{equation}\label{eq:gflasso}
\underset{\beta\in\mR^{(n-1)\times d}}{\text{Minimize}}\; \frac{1}{2}\|\bar{Y}-\bar{D}\beta\|^2_F + \lambda \sum_{i=1}^{n-1}\|\beta_{i,\bullet}\|_2
\end{equation}
with the same ~$\lambda\geq0$, where ~$\bar{Y}$ ~and ~$\bar{D}$ ~are the column-wise centered matrices ~$Y$ ~and ~$D$, respectively. Let ~$\hat{\beta}(\lambda)$ ~denote the solution of \eqref{eq:gflasso}. The solution ~$\hat{U}$ ~of \eqref{eq:totvar} can be recovered via  ~$\hat{U}=\mathds{1} \hat{\gamma} + D\hat{\beta}$ ~with ~$ \hat{\gamma} = \mathds{1}^T(Y-D\hat{\beta})/n$. 
The indices of non-zero rows of matrix ~$\hat{\beta}$ ~correspond to the change point set ~$\cE$ ~in \eqref{eq:set_changes} via \eqref{eq:substitute} and it holds that ~$\cE(\lambda)= \{u \;|\; \hat{\beta}_{u,\bullet}(\lambda)\neq0\}$. 
\\
\\
The crucial observation for further theoretical analysis is that ~$\beta$ ~minimizes \eqref{eq:gflasso}, for any fixed ~$\lambda$, if it fulfills the necessary and sufficient Karush-Kuhn-Tucker (KKT) conditions:
\begin{align}
\begin{split}\label{eq:KKT}
	  \phantom{\|}G_i\phantom{\|} = \lambda B_i&\qquad \forall \beta_{i,\bullet}\neq 0,\\
	  \|G_i\| \leq \lambda\phantom{B_i}&\qquad \forall \beta_{i,\bullet}  = 0,
\end{split}
\end{align}
for all ~$i=1,\ldots,n-1$ ~with vectors ~$G_i=\bar{D}^T_{\bullet,i}\left(\bar{Y}-\bar{D}\beta\right)$ ~and ~$B_i=\beta_{i,\bullet}/\|\beta_{i,\bullet}\|$.
\\
\\
The next proposition formalizes the selection of the regularization parameter ~$\lambda$ ~which was already informally described in the Subsection \ref{sec:single_change_point}. 
\begin{proposition}\label{prop:selection_procedure}
Consider the random matrix ~$\hat{c}=\bar{D}^T\bar{Y}$, set ~$t_i=\|\hat{c}_{i,\bullet}\|$ ~for ~$i=1,\ldots,n-1$. Assume ~$t_{i_1}\leq t_{i_2}\leq \ldots\leq t_{i_{n-1}}$ ~with ~$i_k\neq i_r$ ~for ~$k\neq r$ ~and set ~$M={i_{n-1}}$, $m={i_{n-2}}$. Under Assumption \ref{ass:amoc} it holds that:
\begin{enumerate}
\item If ~$\lambda\geq t_{M}$ ~then ~$\beta=0$ ~solves the KKT system \eqref{eq:KKT}.
\item If ~$t_{m}<\lambda< t_{M}$ ~then a random ~$\lambda_{\min}$ ~exist such that for any ~$\lambda_{\min}<\lambda< t_{M}$ ~the ~$\beta$ ~with rows
\begin{equation}\label{eq:def_beta}
	\beta_{i,\bullet}=
\alpha
\begin{cases}
\hat{c}_{M,\bullet}, & i=M,\\
0, & i\neq M,
\end{cases}
\end{equation}
and ~$\alpha=(t_{M}-\lambda)/(\bar{D}^T_{\bullet,M}\bar{D}_{\bullet,M}t_{M})$ ~solves the KKT system \eqref{eq:KKT}.  
\end{enumerate} 
In the latter case we obtain ~$\cE(\lambda)=\{M\}$, i.e. ~$\hat{u}=M$.
\end{proposition}

\begin{proof}[Proof of Proposition \ref{prop:selection_procedure}]
For ~$\beta=0$ ~the conditions of \eqref{eq:KKT} simplify to ~$\lambda\geq t_{M}$ ~and the first statement follows immediately. We turn to the second statement where ~$t_{m}<\lambda< t_{M}$ ~and in which case
\[
	\bar{D}^T_{\bullet,M}\bar{D}\beta=\bar{D}^T_{\bullet,M}\bar{D}_{\bullet,M}\beta_{M,\bullet}.
\]
Therefore, the first equality in \eqref{eq:KKT} translates to
\begin{align}\label{eq:intersteps}
\begin{split}
	\hat{c}_{M,\bullet} &= \bigg(\frac{\lambda}{\|\beta_{M,\bullet}\|} +  \bar{D}^T_{\bullet,M}\bar{D}_{\bullet,M} \bigg)\beta_{M,\bullet}\\
&=  \bigg(\frac{\lambda}{(t_M-\lambda)/(\bar{D}^T_{\bullet,M}\bar{D}_{\bullet,M})} +  \bar{D}^T_{\bullet,M}\bar{D}_{\bullet,M} \bigg)\beta_{M,\bullet}
\end{split}
\end{align}
which is fulfilled by the definition of ~$\beta_{M,\bullet}$.
Here, the latter equality in \eqref{eq:intersteps} holds true since the former equality in \eqref{eq:intersteps} implies 
\[
	t_M = \|\hat{c}_{M,\bullet}\| = \lambda  + \bar{D}^T_{\bullet,M}\bar{D}_{\bullet,M}\|\beta_{M,\bullet}\|.
\]
For ~$\lambda\uparrow  t_{M}$ ~we have ~$\beta\rightarrow 0$ ~and therefore ~$\|G_{i}\|\rightarrow t_{i}<t_M$ ~for ~$i\neq M$ ~which yields the assertion.
\end{proof} 
We are now in the position to provide the short proof for Proposition \ref{prop:LASSO_CUSUM}.
\begin{proof}[Proof of Proposition \ref{prop:LASSO_CUSUM}]
Straightforward calculations yield
\[
	\left(\bar{D}^T\bar{Y}\right)_{i,k} = -w(i,n)\sum_{j=1}^i (Y_{j,k}-\bar{Y}_{n,k})
\] 
for all ~$i=1,\ldots,n-1$ ~and ~$k=1,\ldots,d$. Hence, we obtain ~$t_i^2=\|\hat{c}_{i,\bullet}\|^2=t(i)$. Since ~$t(i)$ ~has a unique maximum, we know that an appropriate ~$\lambda$ ~with ~$t_m<\lambda< t_{M}$ ~may be selected which, together with Proposition \ref{prop:selection_procedure}, completes the proof.
\end{proof}
We continue with the proofs for Subsection \ref{sec:theoretical}.

\begin{proof}[Proof of Example \ref{example:MA1}]
According to \eqref{eq:partial_sum_rewritten} and due to independence we have
\begin{align*}
	\Var(S_{i,k}(\varepsilon)) &=\Var \left(\sum_{j=1}^n a_{i,j} \left(\eta_{j,k}+\phi\eta_{j-1,k}\right)\right)\\
&+\theta^2\Var \left(\sum_{j=1}^n a_{i,j} \left(\eta_{j,k-1}+\phi\eta_{j-1,k-1}\right)\right).
\end{align*}
Straightforward calculations yield  
\begin{align*}
&\frac{n}{\tilde{\sigma}^2}\Var \left(\sum_{j=1}^n a_{i,j} \left(\eta_{j,k}+\phi\eta_{j-1,k}\right)\right)\\
&=\left(1-i/n\right)^2\left[i(1+\phi^2)+2(i-1)\phi\right]\\
	&\quad+ \left(i/n\right)^2\left[(n-i)(1+\phi^2)+2((n-i)-1)\phi\right]\\ 
	&\quad- 2\left(1-i/n\right)\left(i/n\right)\phi\\
&=(1+\phi^2+2\phi)\left[\left(1-i/n\right)^2i + \left(i/n\right)^2(n-i)\right]\\
&\quad-2\phi\left[\left(1-i/n\right)^2+\left(i/n\right)^2+\left(1-i/n\right)\left(i/n\right)\right]\\
&=\alpha(\phi)\left(i\left(1-i/n\right)\right)-2\phi,
\end{align*}
where ~$\alpha(\phi)$ ~is set in \eqref{eq:alpha_and_sigma} and this implies
\[
		V^2(i)\xi =\alpha(\phi)\left((i/n)\left(1-i/n\right)\right)-2\phi/n
\]
with ~$\xi=\sigma^2/(\tilde{\sigma}^2(1+\theta^2))$. 
\end{proof}

\begin{proof}[Proof of Theorem \ref{thm:convergence_2max}] Using the notation of \eqref{eq:def_tdpartial} we consider
\begin{align*}
	S_{i,k}(Y)&=n^{-1/2}\sum_{j=1}^i (Y_{j,k} - \bar{Y}_{n,k})\\
				&=n^{-1/2}\sum_{j=1}^i (\varepsilon_{j,k} - \bar{\varepsilon}_{n,k})+\gamma_k n^{-1/2}\sum_{j=1}^i (\zeta_{j} - \bar{\zeta}_{n})-n^{1/2}H(i,u)\Delta_k\\
				&=S_{i,k}(\varepsilon)+\gamma_k S_{i}(\zeta)+C_k,
\end{align*}
with non-random ~$C_k=-n^{1/2}H(i,u)\Delta_k$ ~and where ~$S_{i}(\zeta)=n^{-1/2}\sum_{j=1}^i (\zeta_{j} - \bar{\zeta}_{n})$. It holds that
\begin{align*}
	S_{i,k}(Y)^2 &= | S_{i,k}(\varepsilon)+\gamma_k S_{i}(\zeta)|^2 +2 C_k (S_{i,k}(\varepsilon)+\gamma_k S_{i}(\zeta))   + C_k^2
\end{align*} 
and together with the independence of the centered ~$S_{i,k}(\varepsilon)$ ~and ~$S_{i}(\zeta)$ ~we get
\begin{align*}
	E|S_{i,k}(Y)|^2&=E|S_{i,k}(\varepsilon)|^2+\gamma^2_k E|S_{i}(\zeta)|^2 + C_k^2.
\end{align*}  
Due to part 2 of Assumption \ref{ass:common_structure} and due to Assumption \ref{ass:common_factors} it holds that ~$\sum_{k=1}^d E|S_{i,k}(\varepsilon)|^2/d=V^2(i)\sigma^2$ ~and that ~$\sum_{k=1}^d\gamma_k^2/d=o(1)$, as ~$d\rightarrow\infty$. Therefore, we get that, as ~$d\rightarrow \infty$,
\begin{align}\label{eq:asymptotic_expectation}
	E\bigg(\frac{1}{d}\sum_{k=1}^d |S_{i,k}(Y)|^2\bigg)&= V^2(i)\sigma^2+nH^2(i,u)\tilde{\Delta}^2_d + o(1)
\end{align}
with ~$\tilde{\Delta}^2_d=\sum_{k=1}^d\Delta^2_k/d$ ~and for each ~$i=1,\ldots,n-1$. Now, assume that we already know that 
\begin{align}\label{eq:remaining_variation}
\begin{split}
	\Var\bigg(\frac{1}{d}\sum_{k=1}^d |S_{i,k}(Y)|^2\bigg)&=o(1)
\end{split}
\end{align}
for any ~$i$. Then, via Chebyshev's inequality, we obtain from \eqref{eq:asymptotic_expectation} and \eqref{eq:remaining_variation} that,  as ~$d\rightarrow \infty$, 
\[
	\frac{1}{n^2\tilde{\Delta}^2_d}\frac{t(i)}{d}=\frac{w^2(i,n)}{n\tilde{\Delta}^2_d}\frac{1}{d}\sum_{k=1}^d |S_{i,k}(Y)|^2\stackrel{P}{\longrightarrow} C(i;u,n,r)
\]
for each ~$i=1,\ldots,n-1$. Due to the continuous mapping theorem we obtain
\[
	\frac{1}{n^2\tilde{\Delta}^2_d}\frac{\max_{i\not\in S}t(i)}{d}\stackrel{P}{\longrightarrow} \max_{i\not\in S} C(i;u,n,r), 
\]
where ~$S$ ~is set according to \eqref{eq:perfect_estimation_not_the_definition} and which then completes the proof. (We neglect the trivial case of $S^c=\emptyset$.)

It remains to show that \eqref{eq:remaining_variation} holds true for any ~$i$ ~indeed. It is sufficient to show that all terms
%\begin{small}
\begin{align}
	 &\Var\Big(\sum_{k=1}^d S_{i,k}^2(\varepsilon)\Big),\label{eq:a}\\
	 &\Var\Big(S_{i}(\zeta)\sum_{k=1}^d \gamma_k S_{i,k}(\varepsilon)\Big)=E(S^2_{i}(\zeta))\sum_{k,r=1}^d \gamma_k  \gamma_r E(S_{i,k}(\varepsilon)S_{i,r}(\varepsilon)),\label{eq:b}\\
	 &\Var\Big(S_{i}^2(\zeta)\sum_{k=1}^d \gamma^2_k\Big)=\Big(\sum_{k=1}^d \gamma^2_k\Big)^2 \Var\Big(S_{i}^2(\zeta)\Big),\label{eq:c}\\
	 &\Var\Big(\sum_{k=1}^d C_k S_{i,k}(\varepsilon)\Big)=\sum_{k,r=1}^d C_k C_r E\Big(S_{i,k}(\varepsilon)S_{i,r}(\varepsilon)\Big),\label{eq:d}\\
	 &\Var\Big(S_{i}(\zeta)\sum_{k=1}^d \gamma_k C_k\Big)= \left(\sum_{k=1}^d \gamma_k C_k\right)^2 \Var\Big(S_{i}(\zeta)\Big)\label{eq:e}
\end{align} 
%\end{small}
are of order ~$o(d^2)$ ~because the mixed covariance terms can be neglected due to the Cauchy-Schwarz inequality. The fourth moments of any ~$\zeta_i$ ~are finite, hence ~$E(S^2_{i}(\zeta))$ ~and ~$\Var(S^2_{i}(\zeta))$ ~are finite too. The terms \eqref{eq:a}, \eqref{eq:b} and \eqref{eq:d} are of order ~$o(d^2)$ ~which follows from Assumptions \eqref{eq:condition2} and \eqref{eq:condition1} if we take 
\begin{align*}
	\Var\Big(\sum_{k=1}^d S_{i,k}^2(\varepsilon)\Big)&=\sum_{j,l,q,m=1}^n a_{i,j} a_{i,q} a_{i,l} a_{i,m}\sum_{k,r=1}^d \Cov (\varepsilon_{j,k}\varepsilon_{q,k},\varepsilon_{l,r}\varepsilon_{m,r})\\ 
\intertext{and}
	E\Big(S_{i,k}(\varepsilon)S_{i,r}(\varepsilon)\Big)&= \sum_{j,l=1}^n a_{i,j} a_{i,l}\Cov (\varepsilon_{j,k},\varepsilon_{l,r}),
\end{align*}
with ~$a_{i,j}$ ~defined in \eqref{eq:def_aij}, into account. The term in \eqref{eq:c} is of order ~$o(d^2)$ ~in view of \eqref{eq:condition_gammas} and similarly \eqref{eq:e} is of order ~$o(d^2)$ ~via the Cauchy-Schwarz inequality and again due to \eqref{eq:condition_gammas}.
\end{proof}

\begin{proof}[Proof of Theorem \ref{thm:charct_eq}] 
The critical function ~$C(\cdot;u,n,r)$ ~is positive. It has generally a maximum at ~$i=u$ ~for all ratios ~$r>0$ ~if and only if
\begin{equation}\label{eq:basic_inequality}
	1 \leq \frac{C(u;u,n,r)}{C(i;u,n,r)} 
\end{equation}
holds true for all ~$r>0$, ~$i=1,\ldots,n-1$. This can only be fulfilled if
\begin{equation}\label{eq:condition_basicC}
	1\leq\lim_{r\rightarrow\infty}\frac{C(u;u,n,r)}{C(i;u,n,r)} = \left[\frac{w(u,n)V(u)}{w(i,n)V(i)}\right]^2
\end{equation}
holds true for every ~$i=1,\ldots,n-1$ ~and the positive weights ~$w(i,n)=\alpha/V(i)$ ~from \eqref{eq:charct_eq} fulfill these constraints.

If \eqref{eq:charct_eq} does not hold, then either ~$\alpha\leq 0$ ~and thus ~$w(i,n)$ ~is negative or ~$w(i,n)V(i)$ ~is not a constant function in ~$i$ ~and thus
\[
	1<\left[\frac{w(q,n)V(q)}{w(p,n)V(p)}\right]^2
\]
now holds true for some ~$q\neq p$. The former contradicts the positivity constraint and the latter contradicts condition \eqref{eq:condition_basicC}  for ~$u=p$, $i=q$ ~and therefore also \eqref{eq:basic_inequality} for some ~$r>0$. 
\end{proof}

\begin{proof}[Proof of Theorem \ref{thm:cond_v}]
It holds ~$w^2(i,n)V^2(i)\equiv1$. Hence, for any ~$r\geq 0$, a unique maximum of ~$C(i;u,n,r)$ ~is equivalent to
\begin{align*}
	0&<C(u;u,n,r)-C(i;u,n,r)\\
&=
\begin{cases}
w^2(u,n)(u/n)^2(1-u/n)^2-w^2(i,n)(i/n)^2(1-u/n)^2,&i<u,\\
w^2(u,n)(u/n)^2(1-u/n)^2-w^2(i,n)(u/n)^2(1-i/n)^2,&i>u
\end{cases}
\end{align*}
for all $i\neq u$. Due to the symmetry of ~$V(i)$ ~this is equivalent to 
\[
	\frac{V(n-i)}{V(n-u)}=\frac{V(i)}{V(u)}=\frac{w(u,n)}{w(i,n)}>
\begin{cases}
i/u,&i<u,\\
(n-i)/(n-u),&i>u
\end{cases}
\]
and the assertion follows.
\end{proof}

\begin{proof}[Proof of Lemma \ref{lem:concave_func}]
Set ~$f_u(i)=V(i)/V(u)$, $g_u(i)=i/u$ ~and observe that
\[
	f_u(u)=V(u)/V(u)=1=u/u=g_u(u). 
\]
Now, assume that ~$f_u(1)=V(1)/V(u)>1/u=g_u(1)$. Since ~$f_u(i)$ ~is strictly concave in ~$i$ ~and ~$g_u(i)$ ~is linear in ~$i$ ~the assertion follows immediately.
\end{proof}

\begin{proof}[Proof of Remark \ref{rem:verify_ma1}]  
We have to distinguish the two cases ~$\alpha(\phi)>0$ ~and ~$\alpha(\phi)\leq 0$.  In the first case ~$V(i)$ ~is strictly concave and we may use Lemma \ref{lem:concave_func}. Simple calculations show
\begin{align*}
	c\big(V^2(1)-V^2(u)/u^2\big) &= (un)^{-2}(u-1)((u+u\phi^2-2\phi)n+2\phi u)\\
&= (un)^{-2}(u-1)(\phi^2nu-2\phi(n-u)+nu)
\end{align*}
for some ~$c>0$ ~and it is easy to check that ~$V^2(1)-V^2(u)/u^2>0$ ~holds true (for ~$\phi\geq 0$ ~we check via the first equality and for ~$\phi< 0$ ~via the second equality).
The latter case, ~$\alpha(\phi)\leq 0$, occurs only for negative ~$\phi$ ~with 
\begin{equation}\label{eq:region_phi_convex}
	 -1-n^{-1}-(2n^{-1}+n^{-2})^{1/2}  \leq \phi \leq -1-n^{-1}+(2n^{-1}+n^{-2})^{1/2}. 
\end{equation}
In this case it is sufficient to check that 
\[
	h(t):=\Big(\alpha(\phi)\left((t/n)\left(1-t/n\right)\right)-2\phi/n\Big)/t^2
\] 
 is strictly decreasing on ~$t\in [1, n-1]$. Hence, ~$h(i)/\xi=V^2(i)/i^2$ ~with ~$\xi=\sigma^2/(\tilde{\sigma}^2(1+\theta^2))$ ~is strictly decreasing too. It holds that
\[
	n t^3\partial_t h(t)=-\alpha(\phi)t+4\phi
\] 
and a sufficient condition for ~$h(t)$ ~to be strictly decreasing is that ~$-\alpha(\phi)t+4\phi<0$ ~holds true. This condition is fulfilled for any ~$t\in[1,n-1]$ ~whenever it is fulfilled for ~$t=n$. The latter is equivalent to ~$n(\phi^2+2\phi+1)>2\phi$ ~which always holds true since ~$\phi<0$. Finally, ~$V(i)>0$ ~follows immediately from \eqref{eq:partial_sum_rewritten}. 
\end{proof}

\begin{proof}[Proof of Theorem \ref{thm:locationmaxiid}] 
We assume that ~$u> n/2$, i.e. ~$u^*=u$ ~and set ~$w=w^{\text{weighted}}$. (The case ~$u<n/2$ ~follows by symmetry.) We consider the case ~$i\leq u$ ~first and define 
\begin{align*}
	C(x,r)&=F(x)r+G(x)\\
\intertext{with}
	F(x)&=w^2(x,n)V^2(x),\\
	V^2(x)&=(x/n)(1-x/n),\\ 
G(x)&= w^2(x,n)((x/n)(1-u/n))^2
\end{align*}
on ~$[0,n)$, i.e. ~$C(x,r)=C(x;u,n,r)$ ~for ~$x\in\{1,\ldots,u\}$. For convenience we suppress the dependence on ~$u$ ~and ~$n$. It is easy to check that ~$G(x)$ ~is strictly increasing with ~$G(0)=0$, that ~$F(x)$ ~is strictly concave on ~$[0,n]$ ~and symmetrical with respect to ~$x=n/2$ ~and that ~$F(0)=0=F(n)$ ~holds true. Further, ~$C(x,r)$, as a function of ~$x$, has a unique maximum at ~$u$ ~if and only if ~$C(x,r)<C(u,r)$ ~for any ~$x\neq u$. Now, ~$C(x,r)<(>)C(y,r)$, ~$x<y$ ~and ~$n/2<y$ ~is equivalent to 
\[
	r  \begin{cases}
	<(>)R(y,x),& n-y<x<y, \\
	>(<)R(y,x),& 0<x<n-y
\end{cases}
\]
and simply ~$G(x)<(>)G(y)$ ~if ~$x=n-y$,
where
\begin{equation}\label{eq:ratios}
R(y,x)=-\frac{G(y)-G(x)}{F(y)-F(x)}
\begin{cases}
	>0,& n-y<x<y, \\
	<0,& 0<x<n-y.
\end{cases}
\end{equation}
The latter holds true because ~$G(x)$ ~is strictly increasing on ~$[0,y]$, i.e. ~$G(y)-G(x)>0$ ~is strictly decreasing in ~$x$, and because of the properties of ~$F$ ~described above. In the following we assume that ~$y>n/2$ ~and ~$x<y$. Since ~$G(x)<G(y)$ ~we know that ~$C(n-y,r)<C(y,r)$ ~for any ~$r$ ~and since ~$F$ ~is symmetric we also conclude that ~$R(y,x)>R(y,n-x)$ ~for ~$n-y<x<n/2$. Altogether, this implies that ~$C(i,r)<C(u,r)$ ~for all (discrete) ~$0<i<u$ ~and ~$r\geq 0$ ~if and only if 
\[
	0 \leq r< \cR=\min_{n/2\leq i<u}R(u,i).
\]  
The function ~$C(x;u,n,r)$ ~is obviously strictly decreasing for ~$n/2<u<x<n$ ~and the claim follows. 
%(If ~$u=\lceil n/2 \rceil$ ~the function ~$C(i;u,n,r)$ ~has a maximum at ~$u$ ~for any ~$r\geq 0$ ~due to concavity of ~$H(i,u)$ ~and concavity of ~$w(i,n)V(i)$. Hence, any ~$r$ ~is admissible.)
\end{proof}

\begin{proof}[Proof of Theorem \ref{thm:representmaxiid}] 
We restrict our considerations to ~$u>n/2$. The case ~$u<n/2$ ~follows by symmetry. In \citet[Theorem 2]{vert2011a}, i.e. for ~$\gamma=0$, it is used that ~$C(i;u,n,r)$ ~has a global maximum at ~$u$ ~only if ~$C(u;u,n,r)>C(u-1;u,n,r)$. This does not hold true in the case of ~$\gamma\in(0,1/2)$ ~and a global maximum can differ from ~$u$ ~even though ~$C(u;u,n,r)>C(u-1;u,n,r)$ ~holds true. However, we will see that this situation cannot occur if ~$s\in (1/2+ 1/n,B(\gamma)]$. 
\\
\\  
We use the notation from the proof of Theorem \ref{thm:locationmaxiid}. As mentioned there, it is sufficient to consider the case ~$i\leq u$. In this case we know from the proof of Theorem \ref{thm:locationmaxiid} that a possible local maximum of ~$C(x,r)$ ~for ~$x\in[0,u]$ ~can only occur at some ~$x_{\max} \in [n/2,u]$. Moreover, using basic analysis we know that
\[
	\lim_{r\rightarrow \infty}\frac{C(x,r)}{C(y,r)}=\frac{F(x)}{F(y)}
\]
for any ~$0<x,y<n$. Due to strict concavity of ~$F$ ~we know that for any ~$0<\delta<1$ ~it holds that ~$F(n/2)>F(n/2\pm \delta)$. That is, for sufficiently large ~$r$, a local maximum of ~$C(x,r)$ ~occurs within ~$[n/2-\delta,n/2+\delta]\cap [n/2,u]=[n/2,n/2+\delta]$.
\\
\\
Now, we compute the rescaled first derivative of ~$C(x,r)$ ~on ~$[0,n)$ ~which will be denoted by
\begin{align*}
	P(x,r)&:=\frac{((x/n)(1-x/n))^{2\gamma+1}}{x}\partial_x C(x,r).
\end{align*}
$P(x,r)$ ~can be evaluated to a second order polynomial in ~$x$ ~and for ~$x\in(0,n)$ ~we know that ~$\partial_x C(x,r)=0$ ~if and only if ~$P(x,r)=0$. Furthermore, since ~$C(0,r)=0$ ~and ~$C(x,r)\rightarrow \infty$ ~as ~$x\uparrow n$  ~for any ~$r> 0$ ~we may have, in case of ~$\partial_x C(x,r)=0$, either only a saddle point or a maximum and a minimum must occur simultaneously at some ~$0<x_{\max}<x_{\min}<n$. We also know from previous considerations that ~$x_{\max} \geq n/2$.
\\
\\
The discriminant ~$D(r)$ ~of ~$P(x,r)$ ~is a second order polynomial in ~$r$ ~with roots
\[
	r_{1,2}=\frac{-(2\gamma+2)\pm 4  \gamma^{1/2}}{2\gamma -1}(1-u/n)^2.
\]
$D(r)$, as a second order polynomial, must be positive for ~$r>r_2$, where ~$r_2\geq r_1$. Recall that, ~$C(x,r)$ ~has a local maximum within ~$[n/2,n/2+\delta)$ ~for any ~$0<\delta<1$ ~and all sufficiently large ~$r$. Otherwise, ~$D(r)$ ~would be negative for ~$r>r_2$ ~and we would have no extrema of ~$C(x,r)$ ~in case of large ~$r$. 
\\
\\ 
The solution of ~$P(x,r_2)=0$ ~is given by ~$x^*=nB(\gamma)\in (n/2,n)$ ~which is unique and therefore must be a saddle point of ~$C(x,r_2)$. For ~$r_1<r<r_2$ ~a real solution to ~$P(x,r)=0$ ~does not exist and therefore ~$C(x,r)$ ~does not have any extrema on ~$(0,n)$. For ~$0\leq r'\leq r_1$ ~real solutions do exist but the corresponding roots of ~$P(x,r')$ ~cannot correspond to a maximum or a minimum of ~$C(x,r')$ ~on ~$(0,n)$ ~as discussed in the following. Assume that it is not a saddle point, then we would have a maximum and a minimum because they must occur simultaneously at some ~$n/2\leq x_{\max}<x_{\min}<n$, i.e. ~$C(x_{\max},r')>C(x_{\min},r')$. This implies ~$r'\geq R(x_{\max},x_{\min})$ ~and therefore ~$C(x_{\max},r)>C(x_{\min},r)$ ~for any ~$r>r'$, which contradicts the fact of no extrema for ~$r_1<r<r_2$. We did not exclude the possibility of a saddle point at ~$x_s\in(0,n)$ because for our conclusions it won't cause any problems as long as the function remains strictly increasing on ~$(0,n)/\{x_s\}$.
\\
\\
Assume that  ~$\partial_x C(x_0,r_0)=0$ ~for some ~$n/2<x_0<n$, $r_0>0$. For any ~$\varepsilon> 0$ ~we have
\[
	 \partial_x C(x,r_0+\varepsilon)=\partial_x C(x,r_0)+ \varepsilon  \partial_x F(x)
\] 
with ~$\partial_x F(x)<0$ ~for ~$x\in(n/2,n)$ ~and we know that
\[
	\partial_x C(x,r_0)
\begin{cases}
=0,&  x=x_0,\\
>0,&  0<x<n, \quad x\neq x_0 \quad \text{and ~$x_0$ ~is a saddle point}.
\end{cases}
\]
The first equality ensures that ~$\partial_x C(x_0,r_0+\varepsilon)<0$ ~for any ~$\varepsilon>0$ ~and that ~$C(x,r_0+\varepsilon)$ ~has local extrema ~$n/2\leq x_{\max}<x_0<x_{\min}$. The second inequality ensures for saddle points ~$x_0$ ~that, for any ~$0<\delta<1$ ~we can find an ~$\varepsilon>0$ ~such that ~$x_{\max},x_{\min}\in (x_0-\delta,x_0+\delta)$. Moreover, we know that ~$x_{\max}\downarrow n/2$ ~and that ~$x_{\min}\uparrow n$ ~as ~$r\uparrow\infty$.
\\
\\
The properties discussed above ensure that given ~$n/2 +1< u\leq nB(\gamma)$ ~it suffices to compare ~$C(u;u,n,r)$ ~and ~$C(u-1;u,n,r)$ ~to decide whether a maximum is at ~$u$ ~or not. The remaining assertions follow now by simple analysis.
\end{proof}

\begin{proof}[Proof of Proposition \ref{prop:gamma025}] 
As before, we restrict our considerations to ~$u>n/2$ and the case ~$u<n/2$ ~follows by symmetry. Using the notation of Theorem \ref{thm:locationmaxiid} we consider the quantities ~$R(u,x)$, with a continuous argument ~$x\in(n-u,u)$.

    The case ~$\zeta\in(1/2,3/4]$ ~is already shown in \eqref{eq:asymptotic_expression_gam} in Theorem \ref{thm:representmaxiid} and we continue with ~$\zeta\in(3/4,1)$. If ~$nB(\gamma)<u<n$, the properties in the proof of the previous theorem ensure that for sufficiently large ~$n$, for ~$\gamma\in(0,1/2)$, the differentiable function ~$R(u,x)$ ~must have a local minimum at ~$x_{\min}\in[n/2,nB(\gamma)]$ ~such that ~$C(u,r)>C(x,r)$ ~holds true, for any ~$x<u$ ~in case of ~$r<R(u,x_{\min})$ ~and ~$C(u,r)<C(x,r)$ ~holds true for some ~$x<u$ ~in case of ~$r>R(u,x_{\min})$. Some tedious but straightforward calculations for ~$\gamma=1/4$ ~allow us to solve ~$\partial_x R(u,x)=0$ ~explicitly and to identify the minimum at ~$x_{\min}=n/2+(n^2-un)^{1/2}/2$. Since, ~$x_{\min}$ ~is not necessarily in ~$\mN$ ~we consider $R(sn,x_{\min}+\delta)$ for any ~$\delta\in[-1,1]$, where the limits in $n$ equal \eqref{eq:limitfunc}, and (using the mean value theorem) observe that this convergence is uniform in ~$\delta$. Therefore, ~$R(sn,\lfloor x_{\min} \rfloor)$ ~and ~$R(sn,\lceil x_{\min} \rceil)$ ~have the same limits and the assertion follows since ~$\cR(s,1/4)$ ~corresponds to the former or the latter for each ~$n$. Finally, the smoothness properties follow on applying l'H\^{o}pital's rule.
\end{proof}

\section*{Acknowledgment}
The author wishes to thank Prof. J. G. Steinebach for helpful comments and Christoph Heuser for suggestions to the proof of Theorem \ref{thm:representmaxiid}. The author is also thankful for the valuable comments and suggestions of the anonymous referees that helped to improve the quality of this paper. 
This research was partially supported by the Friedrich Ebert Foundation, Germany. 

{\footnotesize  
\begingroup
\setstretch{0.5}

\endgroup
}%end different size
 
\end{document}